\documentclass{amsart}

\title[Non-commutative Fourier-Mukai duality]{A Non Commutative Grauert Theorem and Fourier Mukai Duality for Generalized Complex Tori}
\author{Patrick Antweiler}
\address{Fakultät für Algebra und Zahlentheorie, Universität Hamburg, Germany}
\email{patrick.antweiler@uni-hamburg.de}

\author{Jonathan Block}
\address{Department of Mathematics, University of Pennsylvania, Philadelphia, PA 19104}
\email{blockj@math.upenn.edu}

\usepackage[top=2.5cm,bottom=2.5cm,left=4cm,right=4cm]{geometry}
\usepackage{amsfonts}

\usepackage{graphicx}
\usepackage{float}
\usepackage{amsmath}
\usepackage{amssymb}
\usepackage{amsthm}
\usepackage{bbm}
\usepackage{comment}
\usepackage{hyperref}
\usepackage{enumitem} 
\usepackage[utf8]{inputenc}
\usepackage[T1]{fontenc}
\usepackage{lmodern}
\usepackage{mathrsfs}
\usepackage{changepage}
\usepackage{tikz-cd} 
\usepackage{multirow}
\usepackage[mathscr]{eucal}
\usepackage[toc,page]{appendix}
\usepackage{wasysym}
\usepackage{adjustbox}
\usepackage{blindtext}
\usepackage{makecell}
\usepackage{url}
\usepackage{pgfplots}
\usepackage{tikz}
\usepackage{tikz-3dplot}
\pgfplotsset{compat=1.15}
\usepackage{mathrsfs}
\usepackage{mathtools}
\usetikzlibrary{arrows}
\usepackage{todonotes}
\usepackage{circledsteps}

\DeclareMathOperator{\HH}{HH}

\newcommand{\Wedge}{\scalebox{1.2}{$\wedge$}}
\newcommand{\Dcoh}{D^b_{\mathrm{coh}}}
\newcommand{\A}{\mathcal{A}}

\newenvironment{nalign}{
	\begin{equation}
		\begin{aligned}
		}{
		\end{aligned}
	\end{equation}
	\ignorespacesafterend
}

\theoremstyle{definition}

\newtheorem{def1}{Definition}[subsection]

\newtheorem{remark}[def1]{Remark}
\newtheorem{example}[def1]{Example}

\theoremstyle{theorem}
\newtheorem{lemma}[def1]{Lemma}
\newtheorem{corollary}[def1]{Corollary}
\newtheorem{theorem}[def1]{Theorem}
\newtheorem*{theorem*}{Theorem}
\newtheorem*{proposition*}{Proposition}
\newtheorem{proposition}[def1]{Proposition}

\newtheoremstyle{named}{}{}{\itshape}{}{\bfseries}{.}{.5em}{\thmnote{#3}#1}
\theoremstyle{named}
\newtheorem*{namedtheorem}{}

\begin{document}

\begin{abstract}
    We prove a generalization of Grauert's higher coherence theorem for a class of curved differential graded (non-commutative) Fréchet algebras. This allows us to extend the Fourier-Mukai calculus to derived categories arising in many new contexts. We then apply it and prove equivalence of derived categories of dual generalized complex tori using a non-commutative version of the Poincaré line bundle. This lays the foundation for categories of generalized complex branes on generalized complex tori. Examples include complex tori, symplectic tori as well as their non-commutative and B-field deformations.
\end{abstract}

     \maketitle

	\tableofcontents
	\section{Introduction}
Categories associated to mathematical objects provide powerful invariants. In some cases, the category provides a complete invariant. For example, we have the classic theorem of Gabriel, \cite{Gabriel1962}:
\begin{theorem*}[Gabriel] Two algebraic varieties $X$ and $Y$ are isomorphic if and only if their abelian categories of coherent sheaves $ \text{Coh}(X)$ and $\text{Coh}(Y)$ are equivalent. 
\end{theorem*}
\vspace{0.3cm}

This is the reason why in noncommutative algebraic geometry à la Artin-Zhang and Rosenberg \cite{Rosenberg1998}, a noncommutative variety is defined in terms of  an abelian category with extra structure satisfying some conditions. 

From the category $\text{Coh}(X)$ one can form the category $D^b_{\text{coh}}(X)$, the bounded derived category of coherent sheaves. Objects in this category are bounded complexes of sheaves of $\mathcal{O}_X$-modules with coherent cohomology.
Now it is no longer true that $D^b_{\text{coh}}(X)$ determines $X$. 
But Bondal and Orlov, \cite{BondalOrlov2002} proved the following theorem.
\begin{theorem*} (Bondal - Orlov) If $X$ is a smooth projective variety with ample canonical or anti-canonical class, then $X$ is determined by $D^b_{\text{coh}}(X)$.
\end{theorem*}
\vspace{0.3cm}

But for other varieties it is not true. When it is \underline{not} true, equivalence between $D^b_{\text{coh}}(X)$ and $D^b_{\text{coh}}(Y)$ can underlay some deep and interesting geometrical relationships between $X$ and $Y$.

An early example was provided by Mukai \cite{Mukai1981} who extended the duality of abelian varieties to an equivalence of their derived categories.
Let $X$ be an abelian variety, or more generally, a complex torus $X = V / \Gamma$, where $V$ is a complex vector space, and $\Gamma$ is a lattice in $V$. Classically one defines another complex torus $X^{\vee}$, the dual complex torus, which is the moduli space of degree $0$ holomorphic line bundles $L$ on $X$ (equivalently, $L$ is topologically trivial).

Now there exists a ``universal'' or ``tautological'' Poincar\'{e} line bundle $P$ on $X \times X^{\vee}$, such that for a degree $0$ line bundle $L\in X^\vee$, $P\big|_{X \times \{L\}} \cong L$, and $P\big|_{\{0\} \times X^{\vee}}$ is trivial. 

\vspace{0.2cm}

\begin{theorem*}  (Mukai Duality) 
$ D^b_{\text{coh}}(X) \ni F \longmapsto Rp_{2*}(p_1^* F \otimes^L P) \in D^b_{\text{coh}}(X^{\vee})$ induces an equivalence of categories. 
\end{theorem*}

\vspace{0.2cm}

An exciting aspect of constructing equivalences of derived categories is that it can relate categories of a very different nature. We mention two examples:
\begin{enumerate} \item Beilinson-Bernstein \cite{Beilinson-Bernstein81} establish equivalences of derived categories of representations of a semi-simple Lie algebra $\mathfrak{g}$ with the derived category of $D$-modules on the associated flag variety $\mathcal{B}$, relating the representation theory of $\mathfrak{g}$ to the geometry of $\mathcal{B}$.
\item Kontsevich's homological mirror symmetry conjecture \cite{Kontsevich1995} predicts an equivalence of categories between the derived category of coherent sheaves on a Calabi-Yau manifold and the Fukaya category of its mirror. This relates the complex geometry of one manifold to the symplectic geometry of another in a deep and unexpected way.
\end{enumerate}

\textbf{Deformations of }$D^b_{\text{coh}}(X)$\textbf{.}
 Deformations of a smooth algebraic variety $X$ are classically determined by $H^1(X; T_X)$. Motivated by his work on mirror symmetry, Kontsevich advocated for considering generalized deformations, by  deforming the category $D^b_{\text{coh}}(X)$ (really its dg-enhancement). Here the infinitesimal deformations of $D^b_{\text{coh}}(X)$ 
are controlled by its Hochschild Cohomology,  which one can compute as

\begin{proposition*}\label{Hoch}
\[\begin{array}{lclclcl}
HH^2(D^b_{\text{coh}}(X)) & \cong &H^2(X; \mathcal{O}) &\oplus &H^1(X; T_X) &\oplus &H^0(X; \wedge^2 T_X)\\
\text{Deformations as}&  & \text{Deformations as} & &\text{Deformations as} & &\text{Deformations as} \\
\text{a $A_{[0,\infty)}$-category}& &\text{a gerbe} & &\text{a complex } & &\text{a noncommutative} \\
 & & &&\text{manifold}&&\text{complex manifold} 
\end{array}\]
\end{proposition*}

\noindent In general, a deformation can have components in all three directions.

If one has an equivalence of categories $\Phi: \Dcoh(X) \longrightarrow \Dcoh(Y)$, then it induces an isomorphism on Hochschild Cohomology:
$$ \HH^*\Phi : \HH^*(\Dcoh(X)) \xrightarrow{\cong} \HH^*(\Dcoh(Y)) $$
\textbf{Interesting phenomenon:} $\HH^*\Phi$ can mix the different directions of deformation.
For example, in the case of a complex torus $X = V / \Gamma$:
\[\begin{array}{rcccl}
H^0(X; \Wedge^2 T_X) &\cong &\Wedge^2 V &\cong& H^2(X^\vee; \mathcal{O}) \\
H^1(X; \Wedge^1 T_X) &\cong &V \otimes \bar{V}^\vee& \cong& H^1(X^\vee; \Wedge^1 T_{X^\vee}) \\
H^2(X; \mathcal{O}) &\cong &\Wedge^2 \bar{V}^\vee &\cong& H^0(X^\vee; \Wedge^2 T_{X^\vee})
\end{array}\]
One can show that the equivalence induced by the Poincar\'{e} line bundle $P$ implements the isomorphisms from the left side to the right. So deforming the torus $X$ in a gerby direction will correspond to deforming $X^\vee$ in a non-commutative direction.

Now one can ask if we can implement the correspondences between the deformations by geometrically defined objects. At the level of infinitesimal deformations, and using deformation quantization, this was implemented in the complex torus case by Ben-Bassat, Block and Pantev \cite{BenBassat}. They constructed a gerbe over $X$ determined by an element in $H^2(X,\mathcal{O})$ and a noncommutative, deformation quantization of $X^\vee$ determined by the corresponding element in $H^0(X^\vee;\Wedge^2 \bar{V}^\vee)$.  They also showed that the Poincar\'e line bundle deforms, and implements an equivalence of the derived categories. This was further extended by Arinkin, Block, and Pantev \cite{ABP2013} to a general theorem, between complex manifolds whose derived categories are equivalent through a Fourier-Mukai kernel. But this was still carried out in the formal deformation case. 

The critical point for these infinitesimal deformations is that one can the deformed Poincaré line bundles by using sheaves on spaces. In the non-formal case, one does not have access to sheaf theory, and one has to use global differential geometric constructions to define the deformed objects. 

The second author \cite{Block2006} began a program to extend this theorem to non-formal deformations. The first step taken in that paper was to describe the derived category of a complex manifold $X$, or even better, a dg-enhancement of the derived category, in global differential geometric terms appropriate to the kind of non-formal deformations we are considering. 
    
    Recall from \cite{Block2006} that if $X$ is a compact complex analytic manifold, then the derived category of $X$ can be computed in terms of the Dolbeault algebra $\mathcal A(X) = (\Omega^{0,*}(X),\wedge,\bar \partial, 0)$:
	\begin{theorem*}[\cite{Block2006}]
		There is an equivalence of triangulated categories
		\[ H^0(\mathcal P_{\mathcal A(X)}) \cong D^b_{coh}(X).\] 
	\end{theorem*}	
	$\mathcal P_A$ in the above theorem denotes the \emph{dg-category of cohesive modules over $A$}. It is the full dg-subcategory of all dg-modules over $A$ whose underyling graded $A^\#$-modules are finitely generated and projective. In the case of $\mathcal A(X)$, objects in this category are equivalently described as complexes of smooth vector bundles with flat $\bar \partial$-superconnections. The theorem is thus a generalization of the Koszul-Malgrange theorem which states that a holomorphic vector bundle is equivalently specified by a smooth vector bundle together with a flat $\bar \partial$-connection on it.

    By \cite{cHH}, the continuous Hochschild cohomology of $\mathcal A(X)$ is equal to the Hochschild cohomology of $X$, 
    \[ \text{HH}_{cont}(\mathcal A(X),\mathcal A(X)) = \text{HH}(X) := \text{Ext}_{\mathcal O_{X \times X}}(\Delta_* \mathcal O_X, \Delta_* \mathcal O_X) \cong H^*(\Wedge^* TX).\]
    Hence, continuous deformations of $\mathcal A(X)$ are controlled by the proposition above. An element of $H^2(X, \mathcal O_X)$ is represented by a  $B$-field $B^{0,2}\in \A^{0,2}(X)$ with $\bar\partial B^{0,2}=0$. Using the exponential sequence of sheaves, 
    \[0\to\mathbb{Z}\to \mathcal O_X\to \mathcal O_X^\times\to 0\]
    we get an element $e^{2\pi i B^{0,2}}\in H^2(X,\mathcal O_X^\times)$, thus an $\mathcal O_X^\times$-gerbe. This gives rise to the derived category $D^b_{coh}(X,B)$ of twisted sheaves with coherent cohomology. On the other hand, we obtain the curved graded algebra $\mathcal A(X,B) = (\Omega^{0,*}(X),\wedge, \bar \partial, 2\pi i B^{0,2})$. We have
    \begin{theorem*}[\cite{Block2006},\cite{BlockDaenzer2010},\cite{antweiler}]
		There is an equivalence of triangulated categories
		\[ H^0(\mathcal P_{\mathcal A(X,B)}) \cong D^b_{coh}(X,B).\] 
        where $D^b_{coh}(X,B)$ denotes the derived category of $\exp(2\pi i B^{0,2})$-twisted sheaves with coherent cohomology.
	\end{theorem*}	
    
     $H^1(TX)$ is the space of (first order) Kodaira-Spencer deformations. Small integrable Kodaira-Spencer deformations in $H^1(TX)$ can be interpreted as twisting the $\bar \partial$-operator. The remaining space $H^0(\Wedge^2 TX)$ is more difficult to understand in general. While the formal deformation theories of $\mathcal A(X)$ and $X$ as a generalized complex manifold (as defined in \cite{gualtieri}) match \cite{cHH}, it is not known whether any holomorphic poisson bivector does give rise to a (possibly $A_\infty$) non-commutative deformation of $\mathcal A(X)$. However, this \emph{is} possible for a torus $V / \Gamma$ and a (constant) bivector $\beta \in \Wedge^2 V$. Essentially, the deformed product is determined by
    \[ u_{\gamma} \star u_{\gamma'} := e^{\pi i \beta(\gamma,\gamma')} u_{\gamma + \gamma'}\]
    for $u_\gamma(x) = e^{2\pi i \gamma(x)}$ and $\gamma \in \Gamma^\vee$. $\Gamma^\vee$ denotes the dual lattice $\Gamma^\vee = \{ \gamma \in V^\vee |\, \gamma(\Gamma) \subseteq \mathbb Z \}$.  The $\bar\partial$-operator remains undeformed. 
    In that way, obtain a non-commutative Dolbeault algebra
    \[ \mathcal A(V/\Gamma,\beta,0) := (\Omega^{0,*}(V/\Gamma),\star,\bar \partial, 0). \]\\

   \noindent
   \textbf{Contents of the paper.}
   \newline  
   \emph{Properness and Grauert direct image theorem.}
    One key observation in this paper is that the curved dg-algebras $\mathcal A(X,0,B)$ (for general compact $X$) and $\mathcal A(V/\Gamma,\beta,B)$ (for tori and conjecturally for more general non-commutative manifolds) come equipped with very favorable functional analytic properties, making them into \emph{Fréchet-Fredholm cdg-algebras} ('c' standing for curved), Definition \ref{DefFreFred}. We note that an algebra concentrated in degree $0$ is Fréchet-Fredholm if and only if it is finite-dimensional. In this sense our theory is a generalization of the theory of finite-dimensional algebras. Our first main result is that for a Fréchet-Fredholm cdg-algebra $A$ we have that the dg-category $\mathcal P_A$ is proper:
    \begin{namedtheorem}[Proposition \ref{PropProper}]
		Let $A$ be a Fréchet-Fredholm cdg-algebra. Then the morphism spaces in the homotopy category of $\mathcal P_A$ are finite-dimensional. In fact, the morphism complexes in $\mathcal P_A$ are Fréchet-Fredholm complexes and are thus continuously homotopy equivalent to their finite-dimensional cohomology.
    \end{namedtheorem}
    Let us mention that we do \emph{not} in general expect $\mathcal P_A$ to have a compact generator.
    To a Fréchet-Fredholm cdg-algebra $A$ we can associate Banach algebras $\prescript{}{i}{A}$ by completing $A$ at certain seminorms indexed by $i \in \mathbb N$. Using the Banach implicit function theorem and ideas from \cite{Uhlenbeck}, we obtain the following regularity result:
    \begin{namedtheorem}[Theorem \ref{Regularity}]
		For $i \in \mathbb N$ large enough, the functors $\mathcal P_{A} \to \mathcal P_{\prescript{}{i}{A}}$ are dg quasi-equivalences. In fact, for any two objects $M$ and $N$ of $\mathcal P_{A}$, the map
		\[ \text{Hom}^*_{A}(M,N) \to \text{Hom}^*_{\prescript{}{i}{A}}( \prescript{}{i}{A} \otimes_{A} M  ,\prescript{}{i}{A} \otimes_A N  )\]
		is a continuous homotopy equivalence.
    \end{namedtheorem}
    The regularity result allows us to apply Fredholm theory for Banach modules over Banach algebras in proving the following Grauert-type theorem:
    \begin{namedtheorem}[Theorem \ref{GrauertPrinciple}]
        Let $A,B$ and $C$ be Fréchet-Fredholm cdg-algebras.
        Let $M \in \mathcal P_{A \hat \otimes B^{op}}$, $N \in \mathcal P_{B \hat \otimes C^{op}}$. Then their completed projective tensor product
        \[ M \hat \otimes_B N\]
        is continuously $A \hat \otimes C^{op}$-linearly homotopy equivalent to a cohesive module $T \in \mathcal P_{A \hat \otimes C^{op}}$. In particular $(-) \hat \otimes_B (-)$ induces an $A_\infty$-functor
        \[ \mathcal P_{A \hat \otimes B^{op}} \times \mathcal P_{B \hat \otimes C^{op}} \to  \mathcal P_{A \hat \otimes C^{op}}.\]
    \end{namedtheorem}
    We then consider Fréchet-Fredholm cdg-algebras $A$ which are smooth in a certain sense. In particular for such algebras, there is a cohesive module $\Delta \in \mathcal P_{A \hat \otimes A^{op}}$ representing the identity functor $\mathcal P_{A} \to \mathcal P_{A}$. These include the examples discussed above. The classical Grauert result for compact manifolds is then deduced from the following Proposition (applied to $\mathcal A(g) : \mathcal A(Y) \to \mathcal A(X)$ for $g : X \to Y$)
    \begin{namedtheorem}[Proposition \ref{PropositionGrauertMaps}]
		Let $A$ and $B$ be Fréchet-Fredholm cdg-algebras and let $A$ be Fréchet twist-smooth. Let
		\[ f : A \to B\]
		be a continuous map of cdg-algebras. Then $f$ induces an $A_\infty$-functor 
		\[ \mathcal P(f)^* : \mathcal P_{B} \to \mathcal P_{A}\]
		$A_\infty$-right-adjoint to the extension of scalars functor $\mathcal P(f)_* = B \hat \otimes_{A} (-) : \mathcal P_A \to \mathcal P_B$.
    \end{namedtheorem}

    \emph{Mukay duality.}
    In the last section we apply our framework to the case of non-commutative/curved tori. To state the main result, we need to fix some data fulfilling non-linear relations:
    \begin{itemize}
        \item A torus $V/\Gamma$.
        \item Complex structures $I_1 : V \to V$, $I_2 : V^\vee \to V^\vee$.
        \item Real bivectors $\beta_1 \in \Wedge^2 V$ and $\beta_2 \in \Wedge^2 V^\vee$ (viewed as a bivector on $V^\vee$).
	\item $2$-forms: $B_1 \in \Wedge^2 V^\vee$ and $B_2 \in \Wedge V^2$
	\item The map $\xi := 1-\beta_2^\sharp \beta_1^\sharp  : V^\vee \to V^\vee$ is invertible. 
   
	\item We have \begin{nalign}\label{relationsBetaB}
		&I_2 \xi= - \xi I_1^* \\
		&(B_1^\flat)^{(0,2)_{I_1}} = (\beta_2^\sharp (\xi^*)^{-1})^{(0,2)_{I_1}} \\
		&(B_2^\flat)^{(0,2)_{I_2}} = (\beta_1^\sharp \xi^{-1})^{(0,2)_{I_2}}
	\end{nalign}
    \end{itemize}
    For bivectors $\beta$ we use the notation $\beta^\sharp : V^\vee \to V, \alpha \mapsto \iota_{\alpha}{\beta}$. Similarly, if $B$ is a 2-form we write $B^\flat : V \to V^\vee, v \mapsto \iota_{v} B$. Algebraically, to this data we associate two non-commutative curved algebras:
    \[ \mathcal A(I_1,\beta_1,B_1) \text{ and } \mathcal A(I_2,\beta_2,B_2).\]
    In \cite{DCatsTori}, the authors give a geometric interpretation of the set of equations \ref{relationsBetaB} in terms of T-duality and \emph{generalized complex geometry} \cite{gualtieri} which we now describe. Given a (translation invariant) generalized complex structure $\mathcal J$ on $V/\Gamma$ one can define its T-dual (living on $V^\vee/\Gamma^\vee$) with respect to the canonical pairing between $V$ and $V^\vee$. Concretely, it is given by the formula
    \[ \varphi(\mathcal J) = \begin{pmatrix}
        0 & 1 \\ 1 & 0
    \end{pmatrix} \mathcal J \begin{pmatrix}
        0 & 1 \\ 1 & 0
    \end{pmatrix} : V^\vee \oplus V \to V^\vee \oplus V.\]
    Denote by $\mathcal J_{I_1}$ the generalized complex structure
    \[ \mathcal J_{I_1} = \begin{pmatrix}
        I_1 & 0 \\ 0 & -I_1^*
    \end{pmatrix}\]
    and $\mathcal J_{I_1,\beta_1,B_1} := e^{\beta_1} e^{B_1} \mathcal J_{I_1} e^{-B_1} e^{-\beta_1}$.
    Then the relations \ref{relationsBetaB} are equivalent to 
    \[ \varphi(\mathcal J_{I_1,\beta_1,B_1}) = \mathcal J_{I_2,\beta_2,B_2}.\]
    As pointed out in \cite[Theorem 6.6]{cavalcanti2007generalized}, T-duality can be reinterpreted in terms of the existence of a space filling generalized complex brane on $\mathcal J_{I_1,\beta_1,B_1} \times \mathcal J_{I_2,-\beta_2,-B_2}$. A version of our main result reads:
    \begin{namedtheorem}[Theorem \ref{ThmNCFM}]
    	Let $P_1$ and $P_2$ be the non-commutative Poincaré line bundles over $\mathcal A(I_1 \times I_2,\beta_1 - \beta_2,B_1-B_2)$ and $\mathcal A(I_1 \times I_2,\beta_1 - \beta_2,B_1-B_2)^{op}$.
    	Then $P_1$ and $P_2[-n]$ establish mutually inverse $A_\infty$-quasi-equivalences of dg-categories
    	\begin{align*}
    		&P_1 \hat \otimes (-) : \mathcal P_{\mathcal A(I_1,\beta_1,B_1)} \to \mathcal P_{\mathcal A(I_2,\beta_2,B_2)} \\
    		&P_2[-n] \hat \otimes (-) : \mathcal P_{\mathcal A(I_2,\beta_2,B_2)} \to \mathcal P_{\mathcal A(I_1,\beta_1,B_1)}.
    	\end{align*} 
    \end{namedtheorem}
    In fact, our result is stronger than just a statement about $A_\infty$-quasi-equivalences since the functors are induced by cohesive kernels. We suggest that smooth Fréchet-Fredholm cdg-algebras should be studied up to such equivalence. Note that if $\beta_i = B_i = 0$, then we obtain the usual Fourier-Mukai duality $D^b_{coh}(V/\Gamma,I) \cong D^b_{coh}(V^\vee/\Gamma^\vee,-I^*)$. For $B_1 = 0$ we get \[\mathcal P_{\mathcal A(I,\beta,0)} \cong \mathcal P_{\mathcal A(-I^*,0,B_{\beta})} \]
    where $B_{\beta}^\flat := \beta_1^\sharp$. This equivalence was announced by the second author in \cite{Block2006}. As pointed out earlier, the latter category can be computed in terms of twisted sheaves. We also obtain a strengthening of a result by \cite{polishchuk2003categories} who compare derived categories of non-commutative elliptic curves with their commutative counterparts.

    The comparison of the algebraic result with the geometric analogues suggests that $\mathcal P_{\mathcal A(I_1,\beta_1,B_1)}$ should be thought of as the derived category of the generalized complex structure $\mathcal J_{I_1,\beta_1,B_1}$. In that picture, the Poincaré line bundle $P_1$ realizes the space filling brane on $\mathcal J_{I_1,\beta_1,B_1} \times \mathcal J_{I_2,-\beta_2,-B_2}$.  In fact, as will be argued in \cite{DCatsTori}, Theorem \ref{ThmNCFM} implies that $\mathcal P_{\mathcal A(I_1,\beta_1,B_1)}$ only depends on the generalized complex structure $\mathcal J_{I_1,\beta_1,B_1}$
    rather than on the choices of $I_1,\beta_1,B_1$. Moreover, the definition of the derived category is consistent with the homological mirror symmetry conjecture. 
    \newline
    \newline
    \textbf{Relation with other work.}
    Non-formal non-commutative Fourier-Mukai duality (in a special case) was originally studied by the second author in \cite{Block}. In this paper, we generalize to arbitrary generalized complex tori and fill in gaps in the proof of \cite{Block}. It seems that the claimed result in loc. cit. only holds upon a further localization of the category of quasi-cohesive modules. Furthermore, in loc. cit., an equivalence between larger categories of quasi-cohesive modules was proposed, whereas the present paper proves that this equivalence restricts to the proper dg-category of cohesive modules using the Grauert-calculus proved in the first part of the paper.\\ 
    
    In \cite{DCatsTori}, the first author and coauthors propose the definition of the derived category of generalized complex tori. The present paper is the key ingredient in proving well-definedness of these categories. Moreover, as will be argued in loc. cit., our main result implies that the proposed category is invariant under the action of $O(2n,2n; \mathbb Z)$ which includes T-duality.\\

\noindent\textbf{Acknowledgments.}
We thank Julian Holstein for insightful conversations and useful comments.
The first author acknowledges funding by the Deutsche Forschungsgemeinschaft(DFG,
German Research Foundation) – SFB 1624 “Higher structures, moduli spaces and integrability” – Projektnummer 506632645.\\

    \noindent\textbf{Notations and conventions.} A \emph{complete locally convex pre-complex $M$} consists of complete locally convex vector spaces $M^i$ and continuous maps $d : M^i \to M^{i+1}$. A \emph{graded complete locally convex space $M$} is a complete locally convex pre-complex for which $d = 0$. If $d^2 = 0$, we say that $M = (M^*, d)$ is a \emph{complete locally convex complex}. For such pre-complexes $M$ and $N$ we define the graded space
    \[ \text{Hom}^i(M,N) := \prod_{n \in \mathbb Z} \text{Hom}_{cont}(M^n, M^{n+i}).\]
    If $M$ and $N$ are complexes, we endow $\text{Hom}^*(M,N)$ with the differential
    \[ d(f) = d_N \circ f - (-1)^{|f|} f \circ d_M\]
    making $(\text{Hom}^*(M,N),d)$ into an ordinary discrete cochain complex.
    The dg-category of complete locally convex complexes (and the graded category of graded complete locally convex spaces) admits the following tensor product
    \[ (M \hat \otimes N)^i := \prod_{a+b = i} M^a \hat \otimes_{\pi} N^b \]
    endowed with the usual Koszul differential. Here $ \hat \otimes_{\pi}$ denotes the completed projective tensor product. It is a crucial fact that $(-) \hat \otimes_{\pi} (-)$ commutes with arbitrary products. 
    
    A \emph{complete locally convex cdg-algebra} $A$ ('c' standing for curved) is a cdg-algebra object in the graded category of graded complete locally convex spaces. Concretely, this means that $A$ consists of $(A^*, \cdot, d, B, 1_A)$
    where $(A^*,d)$ is a complete locally convex pre-complex, $\mu = \cdot : A \hat \otimes A \to A$ and $B \in A^2$ satisfy the quadratic relations:
    \begin{itemize}
        \item (Associativity) $\mu \circ (\text{id}_A \hat \otimes \mu) = \mu \circ (\mu \hat \otimes \text{id}_A)$
        \item (Unitality) $\mu \circ (1_A \otimes \text{id}_A) = \mu \circ (\text{id}_A \otimes 1_A) = \text{id}_A$
        \item (Leibniz rule) $\mu \circ (d \otimes \text{id}_A) - \mu \circ (\text{id}_A \otimes d) = d \circ \mu $
        \item (Curvature controlled by $[B,-]$) $d^2 = \mu \circ (B \otimes \text{id}_A) - \mu \circ (\text{id}_A \otimes B) $
        \item (Bianchi identity) $d(B) = 0$
    \end{itemize}
    A \emph{cdg-module} $M$ over $A$ is a cdg-module object for $A$ in the graded category of graded complete locally convex spaces. This means that $M$ consists of $(M^*, d_M, .)$ where $(M^*, d_M)$ is a complete locally convex pre-complex and $\mu_M = . : A \hat \otimes M \to M$ satisfies
    \begin{itemize}
        \item (Curvature) $d_M^2 = \mu_M \circ (B \otimes \text{id}_M)$
        \item (Associativity) $\mu_M \circ (\mu \otimes \text{id}_M) = \mu_M \circ (\text{id}_A \otimes \mu_M)$
        \item (Unitality) $\mu_M \circ (1_A \otimes \text{id}_M) = \text{id}_M$
    \end{itemize}
    Given cdg-modules $M$ and $N$ over $A$ we define
    \[ \text{Hom}_A^i(M,N) := \{ f \in \text{Hom}^i(M^*,N^*) | f \circ \mu_M = \mu_N \circ (\text{id}_A \otimes f) \}\]
    where the differential is given by $d(f) = d \circ f - (-1)^{|f|} f \circ d$ which squares to $0$ (even if $B \neq 0$). In this way, we obtain the \emph{dg-category $A\text{-Mod}$ of (complete locally convex) $A$-modules}.

	\section{Fredholm theory over Banach algebras}
    \subsection{Banach modules}
    In this section we will fix a unital Banach algebra $A$ (concentrated in degree $0$). A \emph{Banach module over $A$} is a Banach space $M$ together with a continuous action
    \[ A \hat \otimes_{\pi} M \to M\]
    satisfying the usual axioms. We define
    \[ \text{Hom}_A(M,N) = \{ f: M \to N \text{ continuous } | f(a.m) = a.f(m) \text{ for all } a,n \}.\]
    It is a closed subspace of all continuous maps $M \to N$ and hence itself a Banach space. If $M$ is an $A^{op}$-module and $N$ is an $A$-module, we let
    \[ M \hat \otimes_A N := \text{coeq}( M \hat \otimes A \hat \otimes N \rightrightarrows M \hat \otimes N)\]
    computed in the category of Banach spaces.
    \subsection{$A$-nuclear operators}
    
    \begin{def1}
        The space of \emph{$A$-nuclear operators} is the subspace
        \[ \text{Hom}_{A-nuc}(M,N) := \phi(\text{Hom}_A(M,A) \hat \otimes_{A} N) \subset \text{Hom}_{A}(M,N)\]
        where $\phi(f \otimes n)(m) = f(m).n$.
        The space of \emph{algebraically $A$-nuclear operators} is the subspace 
        \[ \text{Hom}_{A-alg.nuc}(M,N) := \phi(\text{Hom}_A(M,A)\otimes_{A} N) \subset \text{Hom}_{A}(M,N)\]
        where $\otimes_{A}$ is the uncompleted tensor-product.
    \end{def1}
    \begin{def1}
        Let $M$ be a Banach module over $A$. An $A$-linear map $f : M \to M$ is called \emph{$A$-Fredholm of index $0$} if it has the form $f = T - N$ where $T$ is invertible and $N$ is $A$-nuclear.
    \end{def1}
    \begin{lemma}
        \label{lemmaAFredholm}
        A map $f : M \to M$ is $A$-Fredholm of index $0$ if and only if $f = T - F$ where $T$ is invertible and $F$ is \emph{algebraically} $A$-nuclear.
        \begin{proof}
            Suppose that $f = T-N$ with $T$ invertible and $N$ $A$-nuclear. The space of algebraically nuclear maps is dense in the space of nuclear map. Thus, we can write $N = \lim_{n \to \infty} N_n$ such that the $N_n$ are algebraically nuclear. Writing
            \[ f = (T-N+N_n) - N_n\]
            we see that $(T-N+N_n)$ is invertible for $n$ large enough as $\text{Hom}_A(M,M)$ is a Banach algebra.
        \end{proof}
    \end{lemma}
    \begin{def1}
        Let $B,B'$ be Banach spaces. Consider the subspace of continuous finite rank maps $\text{Hom}_{fin}(B,B') \subseteq \text{Hom}(B,B') $.
        Its closure inside $\text{Hom}(B,B')$ is denoted by $\text{Hom}_{App}(B,B')$. Its elements are called \emph{approximable}. By definition, a map $f : B \to B'$ is approximable if it is a limit of a sequence of finite rank maps.
    \end{def1}
    \begin{lemma}
        \label{FreeCriterion}
        Let $M = A \hat \otimes V$ be a free $A$-module. Let $f \in A \hat \otimes \text{Hom}(V,V)$ be of the form $f = T - K$ where $T \in A \hat \otimes \text{Hom}(V,V)$ is invertible and 
        $N \in A \hat \otimes \text{Hom}_{App}(V,V)$. Here, $\text{Hom}_{App}(V,V)$ denotes the closed subspace of $\text{Hom}(V,V)$ consisting of approximable maps. Then the image of $f$ in $\text{Hom}_A(M,M)$,
        \[ [f] \in \text{Hom}_A(M,M) \cong \text{Hom}(V,A \hat \otimes V)\]
        is $A$-Fredholm of index $0$.
        \begin{proof}
            We note that $A \hat \otimes \text{Hom}(V,V)$ is a Banach algebra. Similarly to the proof of Lemma \ref{lemmaAFredholm}, we can write
            \[ f = \tilde T - F\]
            where $\tilde T \in A \hat \otimes \text{Hom}(V,V)$ is invertible, $F \in A \otimes \text{Hom}_{fin}(V,V)$ and $\text{Hom}_{fin}(V,V)$ is the space of finite rank maps. Then 
            \[ [f] = [\tilde T] - [F]\]
            where $[\tilde T]$ is invertible and $[F]$ is algebraically $A$-nuclear.
        \end{proof}
    \end{lemma}
	\begin{lemma}
        \label{CriterionAlgNuclear}
	    An $A$-linear map $f : M \to N$ is algebraically $A$-nuclear if and only if it factors through free finite-rank module $M \to V \otimes A \to N$ (with $V$ finite-dimensional).
        \begin{proof}
            If $f$ is algebraically $A$-nuclear, we can find $f_1,...,f_n \in \text{Hom}_A(M,A)$ and $p_1,...,p_n \in N$ such that
            \[ f(m) = f_1(m).p_1 + ... + f_n(m).p_n \]
            for all $m \in M$. Thus, set $V = \mathbb C^n$, define $M \to V \otimes A$ as $(f_1,...,f_n)$ and $V \to N$ as $e_i \mapsto p_i$. The converse is proven similarly.
        \end{proof}
	\end{lemma}
    The following proposition gives useful criteria when a complex of $A$-modules is perfect, i.e., continuously homotopy equivalent to a finite complex of finitely generated projective $A$-modules. Its discrete analogue was proved in \cite[Proposition 1.1]{Quillen1996} and our formulation is a special case of \cite[Proposition 4.11]{BlockDaenzer2010}.
    \begin{proposition}
        \label{PropCharacProj}
        Consider a finite complex $M^* = 0 \to M^0 \to ... \to M^n \to 0$ of Banach $A$-modules. Then the following are equivalent:
        \begin{enumerate}
            \item $M^*$ is continuously homotopy equivalent to a finite complex of finitely generated projective $A$-modules.
            \item $M^*$ is a continuous homotopy direct summand of a finite complex of finitely generated projective $A$-modules.
            \item The identity $1_M: M \to M$ is homotopic to a map which is algebraically $A$-nuclear in every degree.
            \item \label{Cond4} There is a continuous map $H : M^* \to M^{*-1}$ such that $dH + Hd : M^* \to M^*$ is (degreewise) $A$-Fredholm of index $0$.
        \end{enumerate}
        \begin{proof}
                \begin{itemize}
                    \item $(2) \Leftrightarrow (1)$: The discrete analogue of this statement is well-known, see \cite{HomLimits}. Thus it suffices to observe that maps between finitely-generated projective $A$-modules are automatically continuous. 
                \item $(4) \Rightarrow (3):$ In each degree, we may write $dH + Hd = T - F$ where $F$ is algebraically $A$-nuclear by Lemma \ref{lemmaAFredholm} and $T$ is invertible. Then we get 
                \[ d T^{-1} H + T^{-1}H d - d(T^{-1}) H = T^{-1}(dH + Hd) = 1- T^{-1} F\]
                so 
                \begin{align*} dT^{-1}H + T^{-1}Hd &= 1 - (T^{-1}F - d(T^{-1}) H)\\
                 &= 1 - (T^{-1}F - T^{-1} d(T) T^{-1} H).
                \end{align*}
                Observe that $d(T) = d(-F)$ is algebraically $A$-nuclear, hence so is $T^{-1}F - T^{-1} d(T) T^{-1} H$.
                \item $(3) \Rightarrow (2)$: Write $dH + Hd = 1 - F$ where $F$ is algebraically nuclear. By Lemma \ref{CriterionAlgNuclear}, we may factor each $F^i : M^i \to M^i$ (degreewise) through a finitely generated free $A$-module $A \otimes V^i$. From this data, one may construct explicitly a differential on $... \to V^i \otimes A \to V^{i+1} \otimes A \to ...$ and maps $f : M \to (V^* \otimes A, d)$, $g : (V^* \otimes A, d) \to M^*$ such that $g \circ f$ is a homotopic to $1_M$, see the proof of \cite[Proposition 1.1]{Quillen1996}.
                \end{itemize}
            \end{proof}
    \end{proposition}
    \begin{remark}
        \label{RmkEuler}
        Consider the case $A = \mathbb C$. Let $F^*$ be a perfect complex over $\mathbb C$. Thus, we can define its Euler characteristic
        \[ \chi(F) := \sum_{i} (-1)^i \text{dim}(H^i(F)) \in \mathbb Z.\]
        It is stable under perturbing the differential $d$ by a compact operator and it is locally constant in $d \in \text{Hom}^1(F^*, F^*)$.
    \end{remark}
    \begin{corollary}
        \label{CorCharacProj}
        Let $M^* = (A \hat \otimes V^*,d)$ be a bounded complex of free $A$-modules. Suppose that $d$ extends to
        \[ d^l \in A \hat \otimes \text{Hom}^1(V^*, V^*)\]
        and that there is $H \in A \hat \otimes \text{Hom}^{-1}(V^*, V^*)$ such that
        \[ d^l H + H d^l = T - K\]
        where $T \in A \hat \otimes \text{Hom}^0(V^*, V^*)$ is invertible and $K \in A \hat \otimes \text{Hom}_{App}^0(V^*, V^*)$. Then $M^*$ is continuously homotopy equivalent to a finite complex of finitely generated projective $A$-modules.
        \begin{proof}
            This follows from Proposition \ref{PropCharacProj} together with Lemma \ref{FreeCriterion}.
        \end{proof}
    \end{corollary}

	\section{Fréchet-Fredholm cdg-algebras}
	\subsection{Fredholm theory for Fréchet spaces}
	We will be working with the notion of a graded-Fréchet space as defined in \cite{Hamilton} and the notion of tame maps between graded-Fréchet spaces.
	\begin{def1}
		A \emph{graded-Fréchet space} is a Fréchet space $F$ with a choice of increasing seminorms indexed by $n \in \mathbb N_0$. Given two graded Fréchet spaces $F$ and $G$ and a continuous map $f : F \to G$ we say that $f$ is of \emph{tame degree $k \in \mathbb Z$} if there exists a "base" $b \in \mathbb N_0$ and constants $C_b, C_{b+1},...$ such that
		\[ ||f(x)||_{i+k} \leq C_i ||x||_i\]
		for all $i \geq b$. Note that if $f$ is of tame degree $k \in \mathbb Z$ then it also of tame degree $\tilde k \in \mathbb Z$ for all $\tilde k \leq k$. The composition of a map of tame degree $n$ with a map of tame degree $m$ is a map of tame degree $n+m$.
		If $F$ and $G$ are graded-Fréchet space then then their projective completed tensor product $F \hat \otimes G$ is also a graded-Fréchet space where the seminorms lie on the diagonal.
	\end{def1}
	\begin{def1}
		\begin{itemize}
			\item 	Let $F$ be a graded-Fréchet space. We say that \emph{$F$ admits smoothing operators} if 
			there is a family $(t^{\epsilon})_{\epsilon > 0}$ of tame degree $1$ operators with common base $b$ such that for all $i \geq b$ we have that the operators
			\[ \Delta^{\epsilon}_i := p_{i+1 \to i-1} \circ t^{\epsilon} - p_{i \to i-1}: F_i \to F_{i-1}\]
			of Banach spaces fulfill
			\begin{nalign} \label{ApproximationProperty} \Delta^{\epsilon}_i \to 0 \text{ as } \epsilon \to 0.
			\end{nalign}
			\item A \emph{large Fréchet pre-complex} consists of a finite collection $A^i$ $(i \in \mathbb Z)$ of graded-Fréchet spaces and maps $d^i : A^i \to A^{i+1}$ which are of tame degree $-1$ which admit a "Parametrix" $S$ which is a family of maps $S^i : A^i \to A^{i-1}$ which are of tame degree $1$ such that $T - K = dS + Sd$ where
			\begin{enumerate}
				\item $T$ is of tame degree $0$ with tame degree $0$ inverse and,
				\item $K$ is of tame degree $1$.
			\end{enumerate}
			Furthermore, it is required that each $A^i$ admits smooth operators (which may not be compatible with $d$). If $d^2 = 0$. we say that $A$ is a \emph{large Fréchet complex}.
			\item A \emph{Fréchet-Fredholm pre-complex} is a large Fréchet pre-complex such that all localisations $p_{k+1 \to k} : A^i_{k+1} \to A^i_k$ are approximable for all $k$.  If $d^2 = 0$, we say that $A$ is a \emph{Fréchet-Fredholm complex}.
		\end{itemize}
	\end{def1}
    \begin{remark}
	In the definition of Fréchet-Fredholm pre-complexes, we remark that it is actually not necessary to assume the existence of smoothing operators as they can be constructed only using the approximability property.
    \end{remark}

	\begin{example}
		If $F^*$ is a large Fréchet pre-complex and $B$ is any Banach space, then $B \hat \otimes F^*$ is a large Fréchet cdg-complex since the smoothing operators of $F^*$ induce smoothing operators for $B \hat \otimes F^*$. However, if $F^*$ is a Fréchet-Fredholm pre-complex, then unless $B$ is finite-dimensional, $B \hat \otimes F^*$ fails to be a (small) Fréchet pre-complex since a map of the form $1_{B} \hat \otimes T$ where $T$ has finite rank, has infinite-rank unless $T = 0$.
		
		If $(F^*,d)$ is a large Fréchet cdg-complex and $T : F^* \to F^{*+1}$ has tame degree $0$, then $(F^*, d+T)$ is also a large Fréchet cdg-complex (the analogous statement holds for Fréchet-Fredholm pre-complexes).
	\end{example}
	
	\begin{def1}
		\label{defBanachLocalisation}
		Let $A^*$ be a large Fréchet-Fredholm pre-complex. 
		We denote by $A^i_k$ the $k$-th seminorm localisation of the $i$-th degree of $A$. 
		Note that as $d$ is only of tame degree $-1$, we do \emph{not} have maps $A^i_k \to A^{i+1}_k$ so we cannot view $A^*_k$ as a complex. 
		We define the following complexes of Banach spaces for $k$ large enough:
		\[ (\prescript{}{k}{A})^i = A_{k-i}^i\]
		Now $d$ is well-defined on the $\prescript{}{k}{A}$ making each $\prescript{}{k}{A}$ into a complex of Banach spaces.
	\end{def1}
	\begin{lemma}
		Let $A$ be a large Fréchet complex. Then $S$, $T$ and $K$ can be chosen such that $T = \text{id}$.
		\begin{proof}
			It suffices to set $S' := T^{-1} S$ and check that $K' := 1 -dS' - S'd $ is of tame degree $1$.
		\end{proof}
	\end{lemma}
	\begin{lemma}[Bootstrap Lemma]
		\label{LemmaFreBootstrap}
		Let $A^*$ be a large Fréchet complex.
		For $i$ large enough, all maps $\prescript{}{i+1}{A} \to \prescript{}{i}{A}$ are continuous homotopy equivalences, and the maps $A \to \prescript{}{i}{A}$ are continuous homotopy equivalences.
		
		If moreover $A^*$ is a Fréchet-Fredholm complex, then $A$ is also continuously homotopy equivalent to its finite-dimensional cohomology.
		\begin{proof}
			We employ the previous Lemma and so we can assume $T = 1$ (in particular, $K$ is closed).
			The maps $\prescript{}{i+1}{A} \to \prescript{}{i}{A}$ have as homotopy inverse the map $K$ (which was assumed to be of degree $1$) and the homotopy is simply given by $S$. In the case that $A^*$ is small, we can use $S$ as a parametrix and apply Corollary \ref{CorCharacProj} (with respect to the algebra $\mathbb C$). Thus, $\prescript{}{i}{A}$ is continuously homotopy equivalent to its finite-dimensional cohomology groups.
			
			It remains to show that the maps $A \to \prescript{}{i}{A}$ are continuous homotopy equivalences. Given a sequence $(\epsilon_i > 0)_{i \in \mathbb N}$ we define the following maps iteratively. First, set
			\[ F_0 := \text{id}_{A}, H_0 = 0\]
			and then 
			\[ F_{i+1} := (K + d(S t^{\epsilon_{i+1}} )) \circ F_i,  H_{i+1} := H_i + S(1-t^{\epsilon_{i+1}})F_i.\]
			Then each $F_{i}$ is closed and has tame degree $i$ and $H_i$ is of tame degree $1$ fulfilling 
			\[ 1-F_{i} = d(H_i).\]
			By the behaviour of $t^{\epsilon}$ for $\epsilon \to 0$ (see \ref{ApproximationProperty}), we see that the limit of $F_i$ and $H_i$ exist as $i \to \infty$ for some choice of $\epsilon_i$ using that
			\[ K + d(St^{\epsilon_{i+1}}) = \text{id}_A + d(S(t^{\epsilon_{i+1}} - 1)).\]
			The resulting $F$ is of infinite tame degree while $H$ is of tame degree $1$.
			In particular, these operators determine a continuous homotopy equivalence $\prescript{}{b}{A}  \to A$ (where $b$ is the common base of all occuring tame maps).
			
		\end{proof}
	\end{lemma}
	
	\begin{def1}
		Let $F^*$ be a large Fréchet pre-complex. We call a parametrix $S$ for this complex a \emph{strong parametrix} if $S^2 = 0$ and $T = 1$. In that case, we may view $F$ with reversed grading as a large Fréchet complex with differential $S$ and parametrix $d$. So in the case that $F$ is Fréchet-Fredholm one in particular obtains that $\text{Ran}(S)$ is closed.
	\end{def1}
	\begin{corollary}
		\label{CorStrongParametrix}
		Let $F^*$ be a  Fréchet-Fredholm complex. Then $F^*$ admits a strong parametrix.
		\begin{proof}
			By the bootstrap Lemma, $F^*$ is continuously homotopy equivalent to its finite-dimensional cohomology complex $H^*$. Since the latter has vanishing differential, $H^*$ is a strict direct summand of $F^*$, thus $F^* = H^* \oplus C^*$. The component $S_C : C^{*} \to C^{*-1}$ of the homotopy on $F^*$ coming from the homtopy equivalence $F^* \simeq H^*$ is a null-homotopy $1 = dS_C + S_C d$ of tame degree $+1$. To show that $F^*$ admits a strong parametrix, it thus suffices to show that $C^*$ admits one.
			
			We define the following parametrix on $C$:
			\[ H := S_C - d S_C S_C.\]
			We note that $d S_C  S_C d = dS_C (1-dS_C) = dS_C - d(S_Cd)S_C = dS_C - d(1-dS_C)S_C = 0$ as $d^2 = 0$, hence $dS_C S_C$ is $[d,-]$-closed and so $dH + Hd = 1$. Moreover, we verify
			\begin{nalign}
				H^2 &= S_CS_C - dS_CS_CS_C - S_CdS_CS_C + dS_CS_CdS_CS_C \\
				&= S_CS_C - S_CS_C + (dS_CS_Cd)(S_CS_C) = 0
			\end{nalign}
			so the claim is proven.
		\end{proof}
	\end{corollary}
	\begin{def1}
        \label{DefFreFred}
		A \emph{Fréchet-Fredholm cdg-algebra} is a complete locally convex cdg-algebra $A = (A^*, 1_A, d, \mu, B)$ such that
		\begin{itemize}
			\item $(A^*,d)$ is a non-negatively graded Fréchet-Fredholm pre-complex,
			\item $\mu : A \hat \otimes A \to A$ is of tame degree $0$ (with respect to the seminorms $(A \hat \otimes A)_i := A_i \hat \otimes A_i)$.
		\end{itemize}
	\end{def1}
	\begin{example}
		A Fréchet-Fredholm cdg algebra which is concentrated in degree zero is equivalently just a finite-dimensional algebra. 
	\end{example}
	\begin{example}
		Consider the Fréchet algebra of formal power series $\mathbb C[[z^1,...,z^n]]$ where the $i$-th seminorm is the sum of absolute values of coefficients of all monomials of degree $\leq i$. Now consider the formal de-Rham algebra $\mathbb C[[z^1,...,z^n]] \otimes \Wedge^* \mathbb C^n$ endowed with the usual de-Rham differential
		\[ d(f \otimes \omega) := \sum_{i = 1}^n \partial_i f \otimes dz^i \wedge \omega.\]
		This algebra is a Fréchet-Fredholm cdg-algebra with Parametrix
		\[ H(f \otimes \omega) := \frac{(-1)^{p+1}}{p+q} \sum_{i=1}^n f \cdot z^i \otimes \iota_{\partial_i} \omega\]
		for $f$ a homogeneous polynomial of degree $q$ and $p$ is the degree of $\omega$. Its cohomology is $\mathbb C$ concentrated in degree $0$.
	\end{example}
    \subsection{The contraderived category and cohesive modules}
    In this section, we will review two basic categories we associate to a cdg-algebra $A$.
    \begin{def1}
        Let $A$ be a complete locally convex cdg-algebra. We define $\text{Ac}^{ctr}(A)$ to be the minimal thick triangulated subcategory of $A\text{-Mod}$ closed under products which contains the class of absolutely acyclic modules. These are totalisations of sequences 
        \[ 0 \to M \to N \to F \to 0\]
        of cdg-modules which admit continuous $\mathbb C$-linear horizontal splittings. The Verdier quotient by the class of \emph{contra-acyclic} cdg-modules will be denoted by 
        \[ D^{ctr}(A) = H^0(A\text{-Mod}) / \text{Ac}^{ctr}(A).\]
        The resulting category is called the \emph{contra-derived category} of $A$.
    \end{def1}
    Contra-derived categores were introduced in a discrete setting in \cite{posi}. As we see in Theorem \ref{TheoremContraModel}, they actually behave better in our complete locally convex framework. 
    \begin{def1}
    \begin{itemize}
            \item Denote by $A^\#$ the underlying graded algebra of $A$ (where we forget about curvature and differential). A graded module over $A^\#$ is just a cdg-module over $A^\#$ with vanishing differential. Any cdg-module $M$ over $A$ determines a graded module $M^\#$ over $A^\#$ by forgetting differentials.
            \item A \emph{graded-projective} cdg-module over $A$ is a cdg-module $M$ such that its underyling graded $A^\#$-module $M^\#$ is projective. This means that $M^\#$ is a direct summand of a free graded $A^\#$-module which are by definition of the form $A^\# \hat \otimes V$ where $V$ is a graded complete locally convex vector space. Denote by $A\text{-Mod}_{gr.pr.}$ the corresponding full dg-subcategory $A\text{-Mod}$.
            \item A cdg-module $M$ is called \emph{cohesive} if $M^\#$ is a direct summand of some $A^\# \hat \otimes V = A^\# \otimes V$ with $V$ finite-dimensional. The full dg-subcategory of $A\text{-Mod}$ consisting of cohesive cdg-modules will be denoted by $\mathcal P_A$.
        \end{itemize}
    \end{def1}
    \begin{theorem}[Theorem 1.3.8 in \cite{cHH}]
        \label{TheoremContraModel}
        The restriction of the localisation functor $H^0(A\text{-Mod}) \to D^{ctr}(A)$ to graded-projective cdg-modules and graded-free cdg-modules are equivalences of categories 
        \[ H^0(A\text{-Mod}_{gr.free}) \xrightarrow{\sim} H^0(A\text{-Mod}_{gr.pr}) \xrightarrow{\sim} D^{ctr}(A).\]
        In particular, the homotopy category of cohesive modules $H^0(\mathcal P_A)$ is a full subcategory of $D^{ctr}(A)$.
    \end{theorem}
    \begin{lemma}[Homotopy transfer]
        \label{HtpyTransfer}
        Let $A$ be concentrated in non-negative degrees. Let $(N, \mathbb E_0) \in A^0\text{-Mod}_{gr.pr}$ and $M \in A\text{-Mod}_{gr.pr}$.
        Let
        \[ f_0 : (N, \mathbb E_0) \to A^0 \otimes_{A^*} M \]
        be an $A^0$-linear map which is an isomorphism in $D^{ctr}(A^0)$.
        Then there exists a differential $\mathbb E$ on $A^\# \hat \otimes_{A^0} N$ extending $\mathbb E^0$, and a map
        \[ f: (A^\# \hat \otimes_{A^0} N, \mathbb E) \to M\]
        extending $f_0$ which is an isomorphism in $D^{ctr}(A)$.    
        \begin{proof}[Proof sketch]
            For every $i \in \mathbb N$, we consider the map
            \[ (\text{id}_{A^i} \hat \otimes f_0)_* : \text{Hom}^*_{A^0}(N,  A^i \hat \otimes_{A^0} N) \to \text{Hom}^*_{A^0}(N,  A^i \hat \otimes_{A^*} M).\]
            As $N$ is graded-projective, this map is a quasi-isomorphism. We need to find $A^0$-linear maps $f^i : N \to A^i \hat \otimes_{A^*} M$ and $\mathbb E^i : N \to A^i \hat \otimes_{A^0} N$, linear for $i \neq 0$ and fulfilling the Leibniz rule for $i = 1$ such that $f = \sum_i f^i$, $\mathbb E= \sum_i \mathbb E^i$ fulfill \begin{align}\label{closedness} \mathbb E^2 = B \cdot (-) \text{ and } \mathbb F f - f \mathbb E = 0 \end{align} Choose some connection $\nabla : N \to A^1 \hat \otimes_{A^0} N$, not neccesarily flat but fulfilling the Leibniz rule (this is possible as $N$ is projective). Then the choice of $\mathbb E^1 - \nabla$ is $A^0$-linear. The failure of \ref{closedness} can be considered iteratively as cocycles in the cone of $(\text{id}_{A^i} \hat \otimes f_0)_*$. A choice of primitive corresponds to the next step in the construction of $\mathbb E$ and $f$. This process is potentially infinite, but the sums $f = \sum_i f^i$, $\mathbb E= \sum_i \mathbb E^i$ are convergent as we totalize tensor-products using direct products. Finally, one shows that $f$ is a contra-derived isomorphism by testing against graded-projective $A$-modules and using the natural filtrations
            \[ F^i M = A^{\geq i} \hat \otimes_{A^*} M, F^i N = A^{\geq i} \hat \otimes_{A^*} N.\]
        \end{proof}
    \end{lemma}
    \begin{corollary}
        \label{CohesiveIdComplete}
        Let $A$ be a cdg-algebra which is concentrated in non-negative degrees. Then we have that homotopy category of cohesive modules $H^0(\mathcal P_A)$ is idempotent complete.
        \begin{proof}
            Let $M$ be a cohesive module over $A$ together with an idempotent $i$ on it. Then $i$ splits in $D^{ctr}(A)$ as $D^{ctr}(A)$ is triangulated and contains products \cite{HomLimits}. Hence $i$ is realized by a decomposition $M \simeq P \oplus Q$ in $D^{ctr}(A)$. Then $A^0 \hat \otimes_{A^*}^L P$ is a homotopy direct summand of $A^0 \hat \otimes_{A^*} M$ which is a complex of finitely-generated projective $A^0$-modules. By \cite{HomLimits}, homotopy idempotents split in the category of complexes of finitely-generated projective $A^0$-modules. Hence $A^0 \hat \otimes_{A^*}^L P \simeq \tilde P$ where $\tilde P$ is cohesive over $A^0$. By the previous Lemma, $\tilde P$ can be endowed with a superconnection, making it into a cohesive module which is isomorphic to $P$ in $D^{ctr}(A)$. 
        \end{proof}
    \end{corollary}
	\begin{proposition}
		\label{PropProper}
		Let $A$ be a Fréchet-Fredholm cdg-algebra. Then the morphism spaces in the homotopy category of $\mathcal P_A$ are finite-dimensional. In fact, the morphism complexes in $\mathcal P_A$ are Fréchet-Fredholm complexes and are thus continuously homotopy equivalent to their cohomology.
		\begin{proof}
			Let $M$ and $N$ be cohesive modules over $A$. As any cohesive module is a strict summand of a free-cohesive module \cite{RHHolstein}, it suffices to assume that the underlying graded of $M$ has the form $A^\# \otimes V $ and $N^\# = A^\# \otimes W$ for finite-dimensional graded-vector spaces $V$ and $W$. The underlying graded of $\text{Hom}(M,N)$ then has the form
			\[ \text{Hom}(M,N) = (V^\vee \otimes A \otimes W, D).\]
			We can write the differential as 
			\[ D = \text{id} \otimes d \otimes \text{id} + f \cdot (-) + (-) \cdot g =: d + \mathbb E\]
			where $d$ is the derivation of $A$. By assumption we have that the multiplication $\mu$ is of tame degree $0$ and so $K$ has tame degree $0$. 
			
			As $A$ is a Fréchet-Fredholm cdg-complex there exists an operator $S$ of tame degree $1$ such that
			\[ T - dS - Sd =: K\]
			is of tame degree $1$ and $T$ is invertible of tame degree $0$. There are at least 2 ways of making $(V^\vee \otimes A \otimes W, D)$ into a Fréchet-Fredholm complex (and either choice yields the claim of the Lemma):
			\begin{itemize}
				\item Endow $V^\vee \otimes A \otimes W$ with the seminorms 
				\[ (V^\vee \otimes A \otimes W)^k_i = \prod_{a+b+c = k} (V^\vee)^a \otimes A^b_i \otimes W^c \]
				and write
				\[ T - DS -SD = - \mathbb E S - S \mathbb E + K\]
				With the choice of the above seminorms, $- \mathbb E S  -  S \mathbb E $ is of tame degree $1$ (as $\mathbb E$ is induced by multiplication and multiplication has tame degree $0$) and so is $K$.
				\item Endow $V^\vee \otimes A \otimes W$ with the seminorms 
				\[ (V^{\vee} \otimes A \otimes W)^k_i = \prod_{a+b+c = k} (V^\vee)^a \otimes A^b_{i+a+c} \otimes W^c \]
				In this case, write $\mathbb E_0$ for the component $V^\vee\otimes W \to V^\vee \otimes A^0 \otimes W$ and $\mathbb E_{\geq 1} := \mathbb E - \mathbb E_0$ and then
				\begin{align}
					\label{EqDecomposition}
					(T + \mathbb E_0 S + S \mathbb E_0) - DS - SD = -\mathbb E_{\geq 1} S - S \mathbb E_{\geq 1} + K.
				\end{align}
				The right-hand side is of tame degree $1$ and the crucial observation is that $(T + \mathbb E_0 S + S \mathbb E_0)$ is invertible as $T^{-1}(\mathbb E_0 S + S \mathbb E_0)$ is nilpotent (so the geometric series converges). Thus, $S$ is still a parametrix for this complex.
			\end{itemize}
			
		\end{proof}
	\end{proposition}
	\begin{remark}
        \label{RmkQF}
		We will make use of the second choice of seminorms as in the proof. It has the advantage that
		\[ \prescript{}{i}(V^\vee \otimes A \otimes W) = (V^\vee \otimes \prescript{}{i}{A} \otimes W) \]
		in the notation of Definition \ref{defBanachLocalisation}.
		In other words, this means that
		\[ \prescript{}{i}{\text{Hom}_{A}(M,N)} =  \text{Hom}_{\prescript{}{i}{A}}(\prescript{}{i}{M},\prescript{}{i}{N})\]
		where $\prescript{}{i}{M} = \prescript{}{i}{A} \hat \otimes_A M  $ which already proves that the functors $\mathcal P_{A} \to \mathcal P_{ \prescript{}{i}{A}}$ are quasi-fully faithful (in fact we will see that they are dg-equivalences).
	\end{remark}
	
	\begin{def1}
		Let $A$ be a Fréchet-Fredholm cdg-algebra. Let $M, N \in \mathcal P_{A}$. We define their \emph{relative Euler characteristic} as
		\[ \chi(M,N) := \sum_{i} (-1)^i H^i (\text{Hom}^*(M,N)) \in \mathbb Z.\]
		We note the following: Given Maurer-Cartan elements $\mathbb E \in \text{Hom}^1(M,M)$ and $\mathbb F \in \text{Hom}^1(N,N)$ the differential on 
		$\prescript{}{i}{\text{Hom}}^*(M,N)$ differs from $\prescript{}{i}{\text{Hom}}^*((M,d + \mathbb E),(N, d + \mathbb F))$ only by a compact perturbation. Hence by Remark \ref{RmkEuler}, the Euler characteristic only depends on the underlying graded $A^\#$-modules of $M$ and $N$.
	\end{def1}

	\subsection{Examples}
	In this subsection we will consider some examples arising in differential geometry.
	\begin{example}
		Let $M$ be a compact smooth manifold.
		Consider a finite collection $V_0,...,V_n$ of smooth vector bundles on $M$ and suppose we have a complex
		\begin{align}\label{EllipticSequence} 0 \to \Gamma(V_0) \xrightarrow{d_0} ... \xrightarrow{d_{n-1}} \Gamma(V_n) \to 0 \end{align}
		with differentials being first-order differential operators. Such a complex is called \emph{elliptic} if the corresponding symbol sequence
		\[ 0 \to V_0 \xrightarrow{\sigma(d_0)} ... \xrightarrow{\sigma(d_{n-1})} V_n \to 0\]
		is exact. Endowing $\Gamma(-)$ with a choice of Sobolevnorms $H^{i,2}$, the standard elliptic operator theory yields a projection $K$ (typically denoted by $H$ for \emph{Harmonic}) onto the cohomology (of infinite tame degree) and the Green's operator $S$ (typically denoted by $G$) is a tame degree $1$ map fulfilling $dS + Sd = 1 - H$. The Rellich-Lemma for Sobolev spaces yields compactness of the completions $H^{i+1,2}(V_k) \to H^{i,2}(V_k)$, and thus approximability as these spaces are Hilbert. Therefore, in our terminology, complex \ref{EllipticSequence} is a Fréchet-Fredholm complex.
	\end{example}
	\begin{example}
		Let $M$ again be a compact smooth manifold and $(L \to M, [-,-], \rho : L \to TM)$ be a (real/complex) Lie-algebroid on $M$. We consider its Chevalley-Eilenberg algebra $(\Gamma(\Wedge^* L^*), \wedge, d_{CH})$. We suppose that $L$ is elliptic, i.e., that the complex
		\[ 0 \to C^\infty(M) \to \Gamma(L^*) \to ... \to \Gamma(\Wedge^n L^*) \to 0\]
		is elliptic. By the previous example, this complex is a Fréchet-Fredholm complex. Moreover, by the Sobolev multiplication theorem, the product on smooth functions (and thus also on the above algebra)
		extends to continuous operators 
		\[ H^{k,2} \hat \otimes H^{k,2} \to H^{k,2}  \]
		for $k$ large enough. It follows that $(\Gamma(\Wedge^* L^*), \wedge, d_{CH})$ is a Fréchet-Fredholm dg-algebra. In particular, letting $L = TM$, we find that the de-Rham complex $\Omega^*(M)$ of $M$ is a Fréchet-Fredholm dg-algebra.
	\end{example}
	\begin{example}
		Let $X$ be a compact complex analytic manifold. The Dolbeault algebra $\mathcal A(X) := (\Omega^{0,*}(X), \wedge, \bar \partial)$ endowed with Sobolev-seminorms is a Fréchet-Fredholm complex, again by the Sobolev multiplication theorem, Rellich Lemma and the fact that $\bar \partial$ is elliptic.

	\end{example}
	\begin{example}
		Let $M$ be a compact smooth manifold and $(E \to M,\nabla)$ a smooth vector bundle with connection $\nabla$. We can form its "de-Rham pre-complex" $\Omega^*(M,E)$ with differential induced by $\nabla$ and we claim that it is a Fréchet-Fredholm pre-complex. Choose a trivial vector bundle $\underline{\mathbb K^m}$ so that $E$ is a direct summand of $\underline{\mathbb K^m}$ via $i : E \to \underline{\mathbb K^m}$ and $p : \underline{\mathbb K^m} \to E$. On the vector bundle $\underline{\mathbb K^m}$ we have two different connections. We denote by $\nabla_{triv}$ the trivial flat connection on $\underline{\mathbb K^m}$. We also define the following connection: 
		\[ \tilde \nabla(f) := i \nabla p(f) + (ip-1) \nabla_{triv} (ip-1)(f).\]
		With this connection $(E, \nabla)$ is a direct summand of $(\underline{\mathbb K^m}, \tilde \nabla)$ compatible with connections. The de-Rham complex $(\Omega(M,\underline{\mathbb K^m}), \nabla_{triv}) = (\Omega(M),d_{DR})^{\oplus m}$ is Fréchet-Fredholm (as seen above). Now $\tilde \nabla = \nabla^{triv} + \omega \Wedge (-)$ for some $\omega \in \Omega^1(M, \text{End}(\mathbb K^m))$ but multiplication is of tame degree $0$, thus $(\Omega(M,\underline{\mathbb K^m}), \tilde \nabla)$ is also Fréchet-Fredholm pre-complex. Since $(\Omega(M, E), \nabla)$ is a direct summand (compatible with all the structure) of this pre-complex, $(\Omega(M, E), \nabla)$ is also Fréchet-Fredholm.
		
		Moreover, we may form the cdg-algebra $(\Omega^*(M, End(E)), [\nabla,-], \nabla^2)$. The above reasoning applied to $End(E)$ shows that this cdg-algebra is a genuine curved non-commutative Fréchet-Fredholm algebra.
	\end{example}
	
	\subsection{Uhlenbeck regularity}
    This subsection is concerned with regularity of cohesive modules over Fréchet-Fredholm cdg-algebras. Specifically, we will show that every cohesive module $M$ over the completed algebra $\prescript{}{i}{A}$ lifts to a cohesive module over $A$. For example, if $\Omega^*(M)$ is the de-Rham algebra of a compact manifold $M$, our result implies that any $H^{i,2}$-measurable \emph{flat} superconnection lifts to a smooth superconnection (where $i$ is large enough so that the standard Sobolev estimates hold).

    The ideas we will adopt are due to Uhlenbeck \cite{Uhlenbeck}. The main ingredient consists of finding a suitable gauge condition which, together with flatness, imply elliptic estimates, so that the connection can be lifted. To state the essential argument, suppose $(M,d) \in \mathcal P_{A}$ is a cohesive module over $A$. Then consider a perturbation $\mathbb E \in \text{Hom}^1(\prescript{}{i}{M}, \prescript{}{i}{M})$ which satisfies the Maurer-Cartan equation, thus $(\prescript{}{i}{M}, d + \mathbb E)$ is a cohesive module over $\prescript{}{i}{A}$. Now suppose $\mathbb E$ satisfies the \emph{Coulomb gauge condition}
    \[ S(\mathbb E) = 0\]
    where $S$ is a parametrix on $\text{Hom}(M,M)$. We have $d(\mathbb E) + \mathbb E^2 = 0$ by assumption and from $1-K = dS + Sd$ we get
    \[ \mathbb E = \prescript{}{i}{[K(\mathbb E) - S(\mathbb E^2)]}.\]
    Hence, $\mathbb E$, of regularity $i$, lifts to $-K(\mathbb E) - S(\mathbb E^2)$ which has regularity $i+1$.
	\begin{theorem}
		\label{Regularity}
		For $i \in \mathbb N$ large enough, the functors $\mathcal P_{A} \to \mathcal P_{\prescript{}{i}{A}}$ are dg quasi-equivalences. In fact, for any two objects $M$ and $N$ of $\mathcal P_{A}$, the map
		\[ \text{Hom}^*_{A}(M,N) \to \text{Hom}^*_{\prescript{}{i}{A}}( \prescript{}{i}{A} \otimes_{A} M  ,\prescript{}{i}{A} \otimes_A N  )\]
		is a continuous homotopy equivalence.
		\begin{proof}
			The (continuous) quasi-fully faithfulness follows from Proposition \ref{PropProper} and Lemma \ref{LemmaFreBootstrap} (see Remark \ref{RmkQF}).
            To show that the functor is quasi-essentially surjective, it suffices to show that every finitely-generated free twisted-module $(\prescript{}{i}{A} \otimes V , D = d + \mathbb E)$ can be lifted to a cohesive $A$-module by Corollary \ref{CohesiveIdComplete}. Consider the graded $A^\#$-module
			\[ T := A^\# \otimes (V \oplus V[1])^.\]
			Elements in the first summand will be denoted $v$ and those living in the second $w$.
			We make $T$ into a cdg-module over $A$ by viewing it as the $G^+$ construction of the graded module $A^\# \otimes V$. This corresponds to setting the differential equal to 
			\[ D^{G^+}(v,w) := (B \otimes w,v)\]
			and it is extended by the Leibniz-rule. The resulting cdg-module over $A$ will be denoted by $G^+$. The crucial point is that this is a cdg-module for $A$, not just for $\prescript{}{i}{A}$ (in the case that the curvature is equal to $0$, we do not have to do this extra step and instead just consider the auxiliary differential $(d_A \otimes 1)$ on $A \otimes V$). We introduce another differential, now on the completed space $\prescript{}{i}{T} = \prescript{}{i}{A}^\# \otimes (V \oplus V[1])  $ which simply comes from viewing $\prescript{}{i}{T}$ as $T^0 := (\prescript{}{i}{A} \otimes V , D = d + \mathbb E) \oplus (\prescript{}{i}{A} \otimes V[1], D = d + \mathbb E)$. This differential will be denoted by $D_0$. Note that by construction, $(\prescript{}{i}{A} \otimes V , D = d + \mathbb E)$ is a direct summand of $T^0$. 
			
			We are now in the situation of Lemma \ref{LemmaLiftMC} applied  to the cohesive module $G^+$ and the $\prescript{}{i}{A}$-Maurer-Cartan element $D_0 - D^{G^+}$. By the Lemma, $T^0 = (\prescript{}{i}{T}, D_0)$ is $\prescript{}{i}{A}$-gauge equivalent to a cohesive $A$-module of the form $(T, D')$. In particular, $(\prescript{}{i}{T}, D_0)$ lies in the essential image of $\mathcal P_{A} \to \mathcal P_{\prescript{}{i}{A}}$. By the quasi-fully faithfulness, the homotopy idempotent on $(\prescript{}{i}{T}, D_0)$ lifts to a homotopy idempotent on $(T, D')$. As $\mathcal P_{A}$ is homotopy idempotent complete by Corollary \ref{CohesiveIdComplete}, the homotopy idempotent splits, yielding a homotopy lift of $(\prescript{}{i}{A} \otimes V , D = d + \mathbb E)$ to a cohesive module over $A$.
		\end{proof}
	\end{theorem}
	\begin{lemma}
		\label{LemmaLiftMC}
		Let $A$ be a Fréchet-Fredholm cdg-algebra and let $(T,d) \in \mathcal P_{A}$ be a cohesive module over $A$. Set $ \prescript{}{i}{T} := \prescript{}{i}{A} \hat \otimes_{A} T \in \mathcal P_{ \prescript{}{i}{A}}$. Let
		\[ \mathbb E \in \text{Hom}^1_{\prescript{}{i}{A}}(\prescript{}{i}{T}, \prescript{}{i}{T})\]
		be a Maurer-Cartan element, so that $(\prescript{}{i}{T}, d+\mathbb E)$ defines a cohesive module over $\prescript{}{i}{A}$. Then there exists a gauge transformation
		\[ g \in \text{Hom}^0_{\prescript{}{i}{A}}(\prescript{}{i}{T}, \prescript{}{i}{T})^{\times}\]
		such that $\mathbb E^{g} = g^{-1} \mathbb E g + g^{-1} d(g)$ lifts to Maurer-Cartan element in $\text{Hom}^1_{A}(T, T)$.
		\begin{proof}
			The complex $(M,D) := (\text{Hom}^*_{A}((T,d),(T,d)), D = [d,-])$ is a Fréchet-Fredholm complex by Proposition \ref{PropProper}. It thus admits a strong Parametrix $S$ by Corollary \ref{CorStrongParametrix} such that
			\begin{align}
				&\label{firstId} DS + SD = 1 - K \\
				&\label{secondId} S^2 = 0
			\end{align}
			where $K$ and $S$ are of tame degree $1$. Let $D^{\mathbb E} = [d+ \mathbb E,-]$ be the differential on $\text{Hom}^*_{\prescript{}{i}{A}}((\prescript{}{i}{T},d + \mathbb E),(\prescript{}{i}{T},d + \mathbb E))$. From \ref{firstId} we obtain the identity \begin{align}\label{thirdId} D^{\mathbb E} S + S D^{\mathbb E}  = 1-K + [\mathbb E,S] =: 1 - \tilde K \end{align}
			where $\tilde K$ is approximable (since $\mathbb E$ is of tame degree $0$ and $S$ is of tame degree $1$).
			
			We claim that $\prescript{}{i}{\text{Ran}(SD^{\mathbb E} )}$ is closed in $\prescript{}{i}{\text{Ran}(S)} \subset \prescript{}{i}{M}$ and has finite codimension (note that $\prescript{}{i}{\text{Ran}(S)}$ itself is closed as $S$ is a strong parametrix). To see that, consider the restriction
			\[ SD^{\mathbb E}|_{\prescript{}{i}{\text{Ran}(S)}}   : \prescript{}{i}{\text{Ran}(S)} \to \prescript{}{i}{\text{Ran}(S)}.\]
			From \ref{secondId} and \ref{thirdId} we obtain that $SD^{\mathbb E}|_{\prescript{}{i}{\text{Ran}(S)}}  = (1-\tilde K)|_{\prescript{}{i}{\text{Ran}(S)}} $.  Thus, $SD^{\mathbb E} $ is a Fredholm operator on the range of $S$ since $\tilde K$ is compact and hence $\prescript{}{i}{\text{Ran}(SD^{\mathbb E} )}$ is closed in $\prescript{}{i}{\text{Ran}(S)} $ and has finite codimension. 
			
			Now consider the finite-dimensional quotient space $\prescript{}{i}{\text{Ran}(S)} / \prescript{}{i}{\text{Ran}(SD^{\mathbb E} )}$. By the density of $M \to \prescript{}{i}{M}$, we can find $m_1,...,m_n \in M$ such that $[S(m_1)],...,[S(m_n)]$ constitutes a basis for $\prescript{}{i}{\text{Ran}(S)} / \prescript{}{i}{\text{Ran}(SD^{\mathbb E} )}$. We let $C := \langle S(m_1),...,S(m_n) \rangle$ which serves as a complement of $\prescript{}{i}{\text{Ran}(SD^{\mathbb E} )}$, i.e., 
			\[ \prescript{}{i}{\text{Ran}(SD^{\mathbb E} )} \oplus C = \prescript{}{i}{\text{Ran}(S)}.\] Denote by $\pi$ the canonical projection
			\[ \pi : \prescript{}{i}{\text{Ran}(S)}  \to \prescript{}{i}{\text{Ran}(S)} / C \cong \prescript{}{i}{\text{Ran}(SD^{\mathbb E} )} \]
			and consider a neighbourhood $U \subset \prescript{}{i}{\text{Ran}(S)^0}$ of $0$ such that $1+u$ is an invertible element in $\prescript{}{i}{M^\times}$ for all $u \in U$ (we use that $\prescript{}{i}{M}^0$ is a Banach-algebra). Then set \[ \prescript{}{i}{M}^{\times}_{\text{Ran}(S)} := 1 + U\]
			and define the (non-linear) map:
			\[ \Phi : \prescript{}{i}{M}^1 \times \prescript{}{i}{M}^{\times}_{\text{Ran}(S)} \to \prescript{}{i}{\text{Ran}(S)} / C, (\mathbb F, g) \mapsto \pi \circ S(\mathbb E^g - \mathbb F)\]
			where $\mathbb E^g$ is defined by
			\[ \mathbb E^g := g^{-1} D (g) + g^{-1} \mathbb E g\]
			so that $g^{-1} \circ D^{\mathbb E} \circ g = D + \mathbb E^g$. Its derivative at $(\mathbb E,1)$ in the second component is given by
			\[ d\Phi_{(\mathbb E,1)}(0,f) = \lim\limits_{h \to 0} \frac{\pi S(\mathbb E^{1+hf}-\mathbb E)}{h} = \pi \circ SD^{\mathbb E}(f)\]
			which is surjective by construction of $C$ and Fredholm. By the local normal form for maps with surjective Fredholm derivative, there are is a neighbourhood $V \times W \subset \prescript{}{i}{M}^1 \times \prescript{}{i}{M}^{\times}_{\text{Ran}(S)}$ of $(\mathbb E,1)$ such that for any $\mathbb F \in V$, there exists $g \in W$ fulfilling \begin{align} \label{EqPhiL} \Phi(\mathbb F, g) = \pi \circ S(\mathbb E^g - \mathbb F) = 0\end{align}		i.e., $S(\mathbb E^g - \mathbb F) \in C$ which implies that $S(\mathbb E^g - \mathbb F)$ is in the image of $M \to \prescript{}{i}{M}$, by construction of $C$. As $M^1 \to \prescript{}{i}{M}^1$ is dense, we can find $\mathbb F \in M^1$ such that $[\mathbb F] \in V \subseteq \prescript{}{i}{M}^1$. Then choose $g$ fulfilling equation \ref{EqPhiL} and hence
			\[ S(\mathbb E^g) \in \text{Ran}(M^1 \to \prescript{}{i}{M}^1).\]
			Let us introduce the notation $\prescript{}{i}[v] \in \prescript{}{i}{M}$ for $v \in \prescript{}{j}{M}$, $j \geq i$. From equation \ref{firstId}, we find
			\begin{align}\mathbb E^g = SD(\mathbb E^g) + DS(\mathbb E^g) - \prescript{}{i}[K(\mathbb E^g)] 
				&= \prescript{}{i}[ -S((\mathbb E^g)^2 + m - K(\mathbb E^g)]
			\end{align}
			where $m \in M$ is chosen such that $\prescript{}{i}[m] = S(\mathbb E^g)$. Thus, \[\mathbb E^l := -S((\mathbb E^g)^2 + m - K(\mathbb E^g) \in \prescript{}{i+1} M^1 \] is a lift of $\mathbb E$. However, unless the completions are injective, $\mathbb E^l$ may not be a Maurer-Cartan element in $\prescript{}{i+1} M^1$. To fix this, we let
			\[ \tilde{\mathbb E} := \mathbb E^l - S(D \mathbb E^l + (\mathbb E^g)^2) \in \prescript{}{i+1} M^1.\]
			Then $\prescript{}{i}[\tilde{\mathbb E} ] = \prescript{}{i}[ \mathbb E^l ] - S(D\mathbb E^g + \prescript{}{i}[(\mathbb E^g)^2]) = \mathbb E^g - 0$ as $\mathbb E^g$ is Maurer-Cartan (and we used that $S$ commutes with completions). Moreover, it is Maurer-Cartan:
			\begin{align}
				D (\tilde{\mathbb E} ) + \prescript{}{i+1}[\tilde{\mathbb E} ^2] &= D(\mathbb E^l - S(D \mathbb E^l + (\mathbb E^g)^2)) + (\mathbb E^g)^2 \\
				&= D(\mathbb E^l) + SD(\mathbb E^2) - D \mathbb E^l - (\mathbb E^g)^2 + \prescript{}{i+1}[K(D \mathbb E^l + \mathbb E^2 ) ] + (\mathbb E^g)^2 \\
				&= K(D \mathbb E^g+ \prescript{}{i}[(\mathbb E^g)^2]) = 0
			\end{align}
			where we used that \[D(D\mathbb E^l + (\mathbb E^g)^2) = D((\mathbb E^g)^2) = D(\mathbb E^g) \mathbb E^g - \mathbb E^g D(\mathbb E^g) = \prescript{}{i+1}[\mathbb E^3 - \mathbb E^3] = 0 .\]
			Thus $\tilde{\mathbb E}$ is a Maurer-Cartan lift of $\mathbb E^g$. By construction, $g$ can be chosen to be arbitrarily small to the identity. Iteration of the above procedure hence yields the desired gauge transformation and Maurer-Cartan lift.
		\end{proof}
	\end{lemma}

	\begin{remark}
		In the situation of the previous theorem, the quasi-fully faithfulness also holds on the level of graded-free modules over $A$ which are of the form $M^\# = A^\# \hat \otimes B$ where $B$ is a Banach space since for such modules $\text{Hom}^*_{A}(M,N)$ naturally carries the structure of a large Fréchet complex. 
	\end{remark}
	\begin{example}
		Consider again the formal de-Rham algebra $\Omega = \mathbb C[[z^1,...,z^n]] \otimes \Wedge^* \mathbb C^n$. We have
		\[  \prescript{}{0}{\Omega}^i = \Omega_{-i}^i = \begin{cases}
			\mathbb C \text{ if } i = 0 \\
			0, \text{else}.
		\end{cases} \]
		Since the parametrix, multiplication and differential have base $0$, the conclusion of Theorem \ref{Regularity} holds already for $i = 0$ and so
		\[ \mathcal P_{\Omega} \simeq \mathcal P_{\mathbb C} = \text{Perf}(\mathbb C).\] 
	\end{example}

	\subsection{The Grauert principle for Fréchet-Fredholm algebras}
	\begin{def1}
		Let $A$ be a complete locally convex cdg-algebra. A module $M$ is called \emph{contra-perfect} if it is isomorphic to a cohesive module in the contraderived category $D^{ctr}(A)$.
	\end{def1}
	\begin{proposition}
		\label{PerfectCriterion}
		Let $A$ be a Fréchet-Fredholm cdg-algebra and let $M$ be an object in $D^{ctr}(A)$ of the form
		\[ M^\# = A^\# \hat \otimes V\]
		where $V$ is a graded-Fréchet space with approximable completions $V_{i+1} \to V_i$ and the differential on $M$ has the form
		\[ d_M = d_A \hat \otimes 1  + \mathbb E\]
		where $\mathbb E$ extends to $\mathbb E : \prescript{}{i}{V} \to A \hat \otimes \prescript{}{i}{V} $ for each seminorm $i$ large enough on the Fréchet-Fredholm  cdg-complex $V$. Moreover, assume that the $0$-component $\mathbb E^0 :  V \to A^0 \hat \otimes V$ is contained in $\mathbb E_0 \in A \hat \otimes \text{Hom}(\prescript{}{i}{V},\prescript{}{i}{V}) \subset \text{Hom}(\prescript{}{i}{V},A \hat \otimes \prescript{}{i}{V})$ for all $i$ large enough. Furthermore, we assume the existence of an $A^0$-linear parametrix $S \in A^0 \hat \otimes \text{Hom}^{-1}(\prescript{}{i}{V},\prescript{}{i}{V})  $ such that $K = 1-\mathbb E^0S - S\mathbb E^0$ extends to an operator $K \in A^0\hat \otimes  \text{Hom}^{0}(\prescript{}{i}{V},\prescript{}{i+1}{V}) $. In that case $M$ is contra-perfect.
		\begin{proof}
			The strategy is to reduce the claim to Banach modules over Banach algebras and then use the regularity Theorem \ref{Regularity}.
			
			We consider the completed module $\prescript{}{j,i}{M}$ over $\prescript{}{j}{A}$ which is given by
			\[ \prescript{}{j,i}{M} =  \prescript{}{j}{A} \hat \otimes \prescript{}{i}{V} . \]
			The differential is well-defined on this space by assumption. We now show that this module is contra-perfect. By homotopy transfer, Lemma \ref{HtpyTransfer}, it suffices to show that the underlying $\prescript{}{j}{A}^0$ module \[\prescript{}{j}{A}^0 \hat \otimes_{\prescript{}{j}{A}} \prescript{}{j,i}{M}  = \prescript{}{j}{A}^0 \hat \otimes \prescript{}{i}{V}  \] is contra-perfect. This module is a free Banach module over the Banach algebra $\prescript{}{j}{A}^0$ and the Parametrix $S$ for $A \hat \otimes \prescript{}{i}{V}$ induces a parametrix for this module fulfilling the condition of Corollary \ref{CorCharacProj} (using the regularity assumption on $K$ and the approximability of the completions). By the Corollary, it is thus contra-perfect. We conclude that there exists a cohesive module over $\prescript{}{j}{A}$ which is homotopy equivalent to $\prescript{}{j,i}{M}$. By the regularity result \ref{Regularity} there exists a unique lift of this module to a cohesive module over $A$ which we denote by $R$. 
			
			It remains to show that $R$ is homotopy equivalent to $M$. We will do this in two steps. First, we introduce the semi-completed module $\prescript{}{i}{M}$ given by $\prescript{}{i}{M} = A \hat \otimes \prescript{}{i}{V}$ and we claim that the natural map
			\[ M \to \prescript{}{i}{M}\]
			is an $A$-linear continuous homotopy equivalence. By \cite[Proposition 1.2.14]{cHH}, it suffices to show that the underlying $A^0$-modules are contiuously homotopy equivalent. This follows as in the proof of Lemma \ref{LemmaFreBootstrap} using the $A^0$-linear parametrix $S$ for $A \hat \otimes V^0$ (it is important that $A^0$ is graded-Fréchet and $m$-convex for the proof to work).
			
			The final step now consists of showing that the contiuous homotopy equivalence $\prescript{}{j,i}{M} \simeq \prescript{}{j}{R}$ of $\prescript{}{j}{A}$-modules can be lifted to a continuous homotopy equivalence $\prescript{}{i}{M} \simeq {R}$ of $A$-modules. Consider $P,Q \in \{ \prescript{}{i}{M}, R\}$ and the corresponding $\prescript{}{j}{A}$-modules $\prescript{}{j}{P},\prescript{}{j}{Q}$. Then $\text{Hom}^*_{A}(P,Q)$ carries the structure of a large Fréchet dg-complex and its completion at $j$ is equal to 
			\[ \prescript{}{j}{\text{Hom}^*_{A}(P,Q)} = \text{Hom}^*_{\prescript{}{j}{A}}(\prescript{}{j}{P}, \prescript{}{j}{Q}).\]
			Thus, by Lemma \ref{LemmaFreBootstrap} the map \[\text{Hom}^*_{A}(P,Q) \to \text{Hom}^*_{\prescript{}{j}{A}}(\prescript{}{j}{P}, \prescript{}{j}{Q})\] is a $\mathbb C$-linear continuous homotopy equivalence. This argument shows that the full subcategories of $D^{ctr}(A)$ and $D^{ctr}(\prescript{}{j}{A})$ on the objects $\{ \prescript{}{i}{M}, R\}$ and $\{ \prescript{}{j,i}{M}, \prescript{}{j}{R}\}$ are dg quasi-equivalent. Since the objects are homotopy equivalent in the latter category, they must also be homotopy equivalent in the former.
			
		\end{proof}
	\end{proposition}
	
	\begin{theorem}[Grauert principle for Fréchet-Fredholm cdg-algebras]
		\label{GrauertPrinciple}
		Let $A$, $B$ and $C$ be Fréchet-Fredholm cdg-algebras. The image of the restriction of the functor
		\[ (-) \hat \otimes^L_{B} (-) : D^{ctr}(A \hat \otimes B^{op}) \times D^{ctr}(B \hat \otimes C^{op}) \to D^{ctr}(A\hat \otimes C^{op})\]
		to $\mathcal P_{A \hat \otimes B^{op}} \times \mathcal P_{B\hat \otimes C^{op}}$ consists of contra-perfect modules. In particular, the above functor induces a (continuous) $A_\infty$-functor 
		\[ \mathcal P_{A \hat \otimes B^{op}} \times \mathcal P_{B \hat \otimes C^{op}} \to \mathcal P_{A \hat \otimes C^{op}}.\]
		and thus a stable functor of homotopy categories.

		\begin{proof}
			It suffices to show the claim for graded-free modules. Thus let $M$ and $N$ be of the form
			\[ M^\# = A\hat \otimes V \hat \otimes B^{op}, N^\# = B\hat \otimes W \hat \otimes C^{op}\]
			with $V$ and $W$ finite-dimensional. The underlying graded of their tensor product is 
			\[ (M \hat \otimes_B N)^\# = A \hat \otimes (V \hat \otimes B \hat \otimes W) \hat \otimes C^{op}.\]
			It fulfills the condition of Proposition \ref{PerfectCriterion} and is thus contra-perfect (one has a similar decomposition as in equation \ref{EqDecomposition}). The parametrix is simply induced by the parametrix for the algebra $B$.
		\end{proof}
	\end{theorem}

	\section{Smooth Fréchet cdg-algebras and Fréchet twist-equivalences}
	\subsection{Fréchet twist-equivalences}
	We let $A$ be a Fréchet cdg-algebra in this section (not necessarily Fréchet-Fredholm or $m$-convex). It turns out that the contra-derived category, modelled on complete locally convex vector spaces, has too few equivalences for our applications. For example, we will consider sequences of the form
	\begin{align} \label{ExampleEq} 0 \to C^\infty(\mathbb C) \xrightarrow{\cdot z} C^\infty(\mathbb C) \to \mathbb C[[\bar z]] \to 0.\end{align}
	It is exact in the discrete sense but does not split continuously. The naive fix in this situation would be to divide out all algebraically exact sequences, i.e., to work with the so-called \emph{max-exact structure}. However, this approach does not behave well with respect to functors induced by bimodules. For example, if $M$ is exact in the discrete sense, $M \hat \otimes_{A} N$ may not be exact, even if $N$ is graded-projective. Our notion of equivalence, called \emph{Fréchet twist-equivalence}, is motivated by the fact that the sequence \ref{ExampleEq} consists of nuclear Fréchet spaces. Hence its exactness is preserved by tensoring with any Fréchet space.
	\begin{def1}
		\label{defFreTwist}
		\begin{itemize}
			\item An $A$-module $M$ is called \emph{contra-acyclic} if it is isomorphic to $0$ in $D^{ctr}(A)$.
			\item An $A$-module $M$ is called \emph{twist-acyclic} if for all $T \in \mathcal P_{A^{op}}$ we have that the complex $T \hat \otimes_A M$ is (algebraically) acyclic, i.e., if its algebraic cohomology groups are trivial.
			\item An $A$-module $M$ is called \emph{Fréchet twist-acyclic} if $F \hat \otimes_A^L M$ is (algebraically) acyclic for every module $F \in D^{ctr}(A^{op})$ whose underlying space is Fréchet.
			\item An $A$-module $M$ is called \emph{contra-perfect/twist-perfect/Fréchet twist-perfect} if there is a module $T_M \in \mathcal P_A$ and a map $T_M \to M$ whose cone is contra-acyclic/twist-acyclic/Fréchet twist-acyclic.
		\end{itemize}
	\end{def1}
	\begin{def1}
    Denote by $FrTw(A) \subset D^{ctr}(A)$ the full thick triangulated subcategory of $D^{ctr}(A)$ consisting of Fréchet twist-acyclic cdg-modules. Note that $FrTw(A)$ is closed under products. We define the Verdier quotient
		\[ D^{FrTw}(A) := D^{ctr}(A) / FrTw(A).\]
		Furthermore, denote by $D^{FrTw}_{Fré}(A)$ the full subcategory of $D^{FrTw}(A)$ of all cdg-modules which are Fréchet.
    \end{def1}
	\begin{lemma}
        \label{LemmaCohesiveStable}
		Let $M$ be a cohesive cdg-module over $A$. Then the following are equivalent:
		\begin{enumerate}
			\item $M$ is $A$-linearly continuously homotopy equivalent to $0$.
			\item $M$ is contra-acyclic.
			\item $M$ is twist-acyclic.
			\item $M$ is Fréchet twist-acylic.
		\end{enumerate}
		\begin{proof}
			The equivalence of the first two statements holds for any graded-projective $A$-module. Let us now assume that $M$ is twist-acyclic and we need to show that it is also continuously homotopy equivalent to $0$. Consider the right cdg-module $M^\vee := \text{Hom}^*_{A}(M,A)$ (we use that $A$ is a bimodule). By writing $M$ as a graded direct summand of a free $A^\#$-module, we see that $M^\vee$ is cohesive. Thus
			\[ \text{Hom}^*_{A}(M,M) = M^\vee \hat \otimes_{A^*} M \simeq 0,\]
			and so the identity on $M$ is a coboundary.
		\end{proof}
	\end{lemma}
	
	\begin{lemma}
		Let $A,B,C$ be a Fréchet cdg-algebras. Then the derived tensor product $(-) \hat \otimes^L_B (-)$ on the level of contraderived categories restricts to functors
		\begin{align}
			&D^{ctr}_{Fré}(A \hat \otimes B^{op} ) \times D^{FrTw}(B \hat \otimes C^{op} ) \to D^{FrTw}(A \hat \otimes C^{op} ) \\
			&D^{FrTw}(A \hat \otimes B^{op} ) \times D^{ctr}_{Fré}(B \hat \otimes C^{op} ) \to D^{FrTw}(A \hat \otimes C^{op} ) \\
			&D^{FrTw}_{Fré}(A \hat \otimes B^{op} ) \times D^{FrTw}_{Fré}(B \hat \otimes C^{op} ) \to D^{FrTw}_{Fré}(A \hat \otimes C^{op} ) 
		\end{align}
		\begin{proof}
			Let us prove the first statement, the others are analogous. Let $M$ and $N$ be graded-free cdg-modules over $A \hat \otimes B^{op}$ and $B \hat \otimes C^{op}$ such that the fiber of $M$ is Fréchet. We need to show that $M \hat \otimes_B N$ is Fréchet twist-acyclic if either $N$ is Fréchet twist-acyclic, or $M$ is contra-acyclic. In the latter case, $M \hat \otimes_B N$ is also contra-acyclic and hence Fréchet twist-acyclic. In the former case, let $F$ be a graded-free Fréchet cdg-module over $C \hat \otimes A^{op}$. Then
			\[ F \hat \otimes_{A \hat \otimes C^{op}} (M \hat \otimes_B N)  = (F \hat \otimes_{A} M) \hat \otimes_{B \hat \otimes C^{op}} N.\]
			The latter is algebraically acyclic as  $F \hat \otimes_{A} M$ is Fréchet and $N$ is Fréchet twist-acyclic.
		\end{proof}
	\end{lemma}
	
	\begin{lemma}
		\label{LemmaCrit}
		Suppose $A$ is a non-negatively graded Fréchet cdg-algebra and let $M$ be a cdg-module. Then $M$ is Fréchet twist-acyclic if and only if the $A^0$-module $ A^0 \hat \otimes^L_{A} M$ is Fréchet twist-acyclic.
		\begin{proof}
			Suppose that $M$ is Fréchet twist-acyclic. Then for any Fréchet dg-module $N$ over $A^0$ we have
			\[ N \hat \otimes^L_{A^0} (A^0 \hat \otimes^L_{A^*} M) = (N \hat \otimes^L_{A^0} A^0) \hat \otimes^L_{A^*} M .  \]
			and the latter is acyclic by assumption, since $N \hat \otimes^L_{A^0} A^0$ can be chosen to be Fréchet.
			Conversely, suppose that $A^0 \hat \otimes^L_{A} M$ is Fréchet twist-acyclic. Let $F$ be a Fréchet cdg-module over $A$. We can assume that $M$ itself is graded-free and that $F$ is graded-free with Fréchet fibers so that $M^\# =  A^\# \hat \otimes V $, $F^\# =   W \hat \otimes A^\#$ where $W$ is Fréchet. Then
			\[ (F \hat \otimes_{A^*} M)^\# = W \hat \otimes A^\# \hat \otimes V\]
			and we introduce the filtration $(Q^i F \hat \otimes_{A^*} M)^\# = W \hat \otimes A^{\geq i} \hat \otimes V$.
			Its filtered pieces are given by
			\begin{align} Q^i / Q^{i+1} &= (F \hat \otimes_{A^*} A^0) \hat \otimes_{A^0} A^i \hat \otimes_{A^0} (A^0 \hat \otimes_{A^*} M) \\
				&= (F \hat \otimes_{A^*} A^i) \hat \otimes_{A^0} (A^0 \hat \otimes_{A^*} M)
			\end{align}
			Since $F \hat \otimes_{A^*} A^i$ is a Fréchet dg-module over $A^0$, the assumption implies that $Q^i / Q^{i+1}$ is algebraically acyclic. As the filtration is bounded below and complete, the claim follows.
		\end{proof}
	\end{lemma}
    \begin{lemma}
        \label{LemmaFreAcyclic}
        Let $A$ be a Fréchet cdg-algebra. Let 
        \[ ...\to M_{1} \to M_0 \to 0 \]
        be a bounded above complex of Fréchet cdg-modules over $A$ with horizontal differential $d_h$. Assume that for every $i \in \mathbb Z$ and Banach space $B$ the sequences
        \[ ... \to B \hat \otimes_{\mathbb C} M_1^i \to B \hat \otimes_{\mathbb C} M_0^i \to 0 \]
        are algebraically exact. Then we have that $\text{Tot}^{\Pi}(M_*)$ is Fréchet twist-acyclic over $A$.
        \begin{proof}
            Let $F$ a be graded-free Fréchet cdg-module over $A^{op}$, $F^\# = V \hat \otimes A^\#$. We need to show that $F \hat \otimes_{A^*} \text{Tot}^{\Pi}(M_*)$ is algebraically acyclic. It is the product totalisation of a double complex of the form 
            \[ ... \to F \hat \otimes_{A^*} M_1 \to F \hat \otimes_{A^*} M_0 \to 0.\]
            Filtration by rows yields a converging spectral sequence whose $E_1$ page are products of the cohomologies of sequences
            \[ ...\to V^a \hat \otimes M_1^b \to V^a \hat \otimes M_0^b \to 0\]
            for all $a, b \in \mathbb Z$. These sequences are exact by the assumption and the Mittag-Leffler principle, Lemma \ref{ReductionToBanach}.
        \end{proof}
    \end{lemma}
	
	\begin{lemma}[Mittag-Leffler for dg-Fréchet spaces]
		\label{ReductionToBanach}
		Let $M_{n \in \mathbb N}$ be an inverse system of dg-Fréchet spaces with dense maps $M_{n+1} \to M_n$. If all $M_n$ are algebraically quasi-isomorphic to $0$, then their inverse limit $M$ is also algebraically quasi-isomorphic to $0$.
		\begin{proof}
			On each $M_{n}$ we choose a family of increasing seminorms indexed by $\mathbb N$. Let $x \in M$ be a cocycle. Fix some $n$ in $\mathbb N$ and let $y \in M_n$ with $dy = [x]_n$. Furthermore, let $\epsilon > 0$. Choose some $y' \in M_{n+1}$ with $dy' = [x]_{n+1}$. Then $d([y']_{n} - y) = 0$ hence there exists $z$ with $dz = [y']_{n} - y$. By the density of $M_{n+1} \to M_n$ there exists $z_{\epsilon} \in M_{n+1}$ such that 
			\[ ||[d[z_{\epsilon}]_n - d(z)]_k ||_n < \epsilon\]
			for all $k \leq n$. Then $\tilde y := y' - dz_{\epsilon} \in M_{n+1}$ is a primitive for $[x]_{n+1}$ which is $\epsilon$-close to $y$ in every $(M_k)_n$ for $k \leq n$ ($(M_k)_n$ being the n-th completion of $M_k$). Applying this procedure iteratively with $\epsilon := 1/n^2$ yields a primitive for $x$ in the limit $n \to \infty$.
		\end{proof}
	\end{lemma}
	
	\subsection{Smooth Fréchet cdg-algebras and the Grauert principle for algebra maps}
	\begin{def1}
		Let $A$ be a Fréchet cdg-algebra.
		We call $A$ \emph{contra-smooth/twist-smooth/Fréchet twist-smooth} if $A$ is contra-perfect/twist-perfect/Fréchet twist-perfect as a module over $A \hat \otimes A^{op}$, see Definition \ref{defFreTwist}.
	\end{def1}
    Note that $A$ being Fr\'echet twist-smooth does not imply that the dg-category
    $\mathcal P_A$ is smooth in the classical sense.
	\begin{example}
		\begin{enumerate}
			\item If $M$ is a smooth manifold, then $C^\infty(M)$ is contra-smooth (and hence smooth in all the other senses). 
			\item The pseudo-compact algebra $\mathbb C[[z_1,..,z_n]]$ is contra-smooth.
			\item The dual numbers $\mathbb C[\epsilon]/\epsilon^2$ are not twist-smooth (and hence not smooth in all the other senses).
			\item The Dolbeault algebra $\mathcal A(X)$ of a (possibly non-compact) complex manifold $X$ is Fréchet twist-smooth as we will see soon. It seems to not be contra-smooth as the map sequence in \ref{CorBanachStable} is known not to have a \emph{continuous} splitting.
		\end{enumerate}
	\end{example}
	\begin{remark}
		If $A$ is a Fréchet-Fredholm algebra which is also twist-smooth (via some cohesive module $\Delta \in \mathcal P_{\mathcal A \hat \otimes A^{op} })$ then the identity functor on $\mathcal P_{\mathcal A}$ is representable by the perfect bimodule $\Delta$ (using the Grauert principle).
	\end{remark}

	\begin{proposition}[Grauert for maps]
		\label{PropositionGrauertMaps}
		Let $A$ and $B$ be Fréchet-Fredholm cdg-algebras and let $A$ be Fréchet twist-smooth. Let
		\[ f : A \to B\]
		be a continuous map of cdg-algebras. Then there is a functor 
		\[f_{perf}^* : D^{ctr}(B) \to D^{ctr}(A) \]
		which agrees with the restriction-of-scalars functor $f^*$ after localising $D^{ctr}(A) \to D^{FrTw}(A)$. The image of the restriction $f_{perf}^*|_{\mathcal P_B}$ consists of contra-perfect objects (and thus the image of the restriction $f^*|_{\mathcal P_B}$ consists of Fréchet twist-perfect modules). Hence, $f$ induces an $A_\infty$-functor 
		\[ \mathcal P(f)^* : \mathcal P_{B} \to \mathcal P_{A}\]
		$A_\infty$-right-adjoint to the extension of scalars functor $\mathcal P(f)_* = B \hat \otimes_{A} (-) : \mathcal P_A \to \mathcal P_B$.
		\begin{proof}
			By assumption there is a cohesive $A \hat \otimes A^{op}$-module $M_\Delta$ and a Fréchet twist-equivalence $\phi : M_\Delta \to A$. The map $f : A \to B$ makes $B$ into a $A \hat \otimes B^{op}$-module. $\phi$ induces a Fréchet-twist equivalence $B \simeq M_\Delta \hat \otimes_A B  $  of $A\hat \otimes B^{op}$-modules. By writing $M_\Delta$ as a graded direct summand of a free cohesive module, we see that $M_\Delta \hat \otimes_A B$ is cohesive over $A \hat \otimes B^{op}$. The functor $f^*_{perf}$ is then defined as 
			\[ f^*_{perf} := [M_\Delta \hat \otimes_A B] \hat \otimes_B (-) : D^{ctr}(B) \to D^{ctr}(A).\]
			It preserves contraperfect objects by Theorem \ref{GrauertPrinciple}. 
		\end{proof}
	\end{proposition}
	\begin{def1}
        Consider a Fréchet cdg-algebra $A$. We define its \emph{Fréchet-twist Hochschild cohomology} as
        \[ HH^i_{FrTw}(A) := \text{Hom}_{D^{FrTw}(A \hat \otimes A^{op})}(A,A[i]).\]
        In general, there are set-theoretic issues with this definition. However, if $A$ is Fréchet twist-smooth via $\Delta \in \mathcal P_{A \hat \otimes A^{op}}$, then 
       \[ HH^i_{FrTw}(A) = H^i(\text{Hom}_{A \hat \otimes A^{op}}^*(\Delta, \Delta)).\]
    \end{def1}
    \begin{example}
        Let $X$ be a (possibly non-compact) complex manifold. Below we will show that the Dolbeault algebra $\mathcal A(X)$ is Fréchet twist-smooth via some cohesive module $\Delta$. Under the equivalence $D^b_{coh}(X \times X) \simeq H^0(\mathcal P_{\mathcal A(X) \hat \otimes \mathcal A(X)})$, the cohesive module $\Delta$ coresponds to the sheaf $\Delta_* \mathcal O_X$. Hence,
        \[ HH^*_{FrTw}(\mathcal A(X)) \cong \text{Ext}^*_{\mathcal O_{X \times X}}(\Delta_* \mathcal O_X,\Delta_* \mathcal O_X) \cong \bigoplus_{a,b} H^a(\Wedge^b TX). \]
    \end{example}
	
	\subsection{Complex manifolds and the classical Grauert Theorem}
	
	\begin{corollary}
		\label{CorBanachStable}
		For every open subset $U \subset \mathbb C$ containing $0$ and Fréchet space $F$ the following sequence is (algebraically) exact
		\[ 0 \to F \hat \otimes C^{\infty}(U) \xrightarrow{1 \otimes z} F \hat \otimes C^{\infty}(U) \xrightarrow{1 \otimes \bar T} F \hat \otimes \mathbb C[[\bar z]] \to 0 \]
		where $\bar T(f)$ is the anti-holomorphic part of the Taylor expansion of $f$.
		\begin{proof}
            The exactness of 
            \[ 0 \to C^{\infty}(U) \xrightarrow{1 \otimes z} C^{\infty}(U) \xrightarrow{1 \otimes \bar T} \mathbb C[[\bar z]] \to 0 \]
            follows from \cite[Chapter IV, Theorem 1.1]{Malgrange1966}. 
            This sequence consists of nuclear Fréchet spaces. Hence its exactness is preserved after applying $F \hat \otimes (-)$, by \cite[Chapter 1, §3.3, Corollary 1]{Grothendieck}.
		\end{proof}
	\end{corollary}

	\begin{corollary}
		\label{CorSkyscraperIsFreAcyclic}
		For every open subset $U \subset \mathbb C$ containing $0$ the totalisation of the following sequence of $\mathcal A(U)$-modules is Fréchet twist-acyclic:
		\begin{nalign}\label{EqAcyclic} 0 \to \mathcal A(U) \xrightarrow{\cdot z} \mathcal A(U) \xrightarrow{\text{ev}_0} \mathbb C \to 0.
		\end{nalign}
		\begin{proof}
			First of all, there is a contra-derived isomorphism \[\mathbb C[[\bar z]] \otimes \Wedge^* \langle d \bar z \rangle \to \mathbb C\] (this can be seen by considering natural filtrations on both modules). Instead of the sequence \ref{EqAcyclic} we may thus consider 
			\[ 0 \to \mathcal A(U) \xrightarrow{\cdot z} \mathcal A(U) \xrightarrow{\tilde T} \mathbb C[[\bar z]] \otimes \Wedge^* \langle d \bar z \rangle \to 0.\] 
			To show that the totalisation of this sequence is Fréchet twist-acyclic, it suffices to show that
			\[ 0 \to C^{\infty}(U) \xrightarrow{\cdot z} C^{\infty}(U) \xrightarrow{\tilde T} \mathbb C[[\bar z]] \to 0\]
			is algeraically acyclic and stable under tensoring with Banach spaces, by Lemma \ref{LemmaFreAcyclic}. This is indeed the case by Corollary \ref{CorBanachStable}.
		\end{proof}
	\end{corollary}
	
	\begin{lemma}
		\label{LemmaLocalReductionFre}
		Let $X$ be  complex manifold and let $M$ be an object in $D^{ctr}(\mathcal A(X))$. Then $M$ is Fréchet twist-acyclic if and only if for each $x \in X$ there exists an open neighbourhood $U$ of $x$ such that $\mathcal A(U) \hat \otimes_{\mathcal A(X)} M$ is a Fréchet twist-acyclic module over $\mathcal A(U)$.
		\begin{proof}
			First of all, if $M$ is Fréchet twist-acyclic, then $M \hat \otimes_{\mathcal A(X)} \mathcal A(U)$ is also Fréchet twist-acyclic for every $U \subset X$ and so if $M$ is locally Fréchet twist-acyclic we can find a cover $U_i$ of finite dimension such that $M$ is Fréchet twist-acyclic on each finite intersection of the $U_i$. Given any Fréchet module $T$ we have
			\[ T \hat \otimes_{\mathcal A(X)} M \simeq T \hat \otimes_{\mathcal A(X)} \text{\v C}(U_i,M)  \]
			by \cite[Proposition 1.4.5]{cHH}.
			It is filtered by the \v Cech-degree. The filtered pieces of this complex are given by products of acyclic spaces by the assumption. As the filtration is bounded below and complete we obtain the claim.
		\end{proof}
	\end{lemma}
	\begin{proposition}
		\label{PropositionFréchetTwist}
		Let $X$ be a complex manifold and $Y \subset X$ a closed complex submanifold. Then $\mathcal A(Y)$ is Fréchet twist-perfect over $\mathcal A(X)$. In particular, $\mathcal A(X)$ is Fréchet twist-smooth for every complex manifold $X$.
		\begin{proof}
			Denote by $\iota$ the inclusion $\iota : Y \to X$. The sheaf $\iota_* \mathcal O_Y$ is coherent and thus locally quasi-isomorphic to a perfect complex. In fact if one employs the standard Koszul resolution locally we see that $\iota_* \mathcal O_Y$ lies in $D^B_{perf}(\mathcal O_X)$ (the subcategory of $D^b_{coh}(X)$ of globally bounded perfect complexes). Denote by $\underline{\mathcal A_Y}$ the sheaf of Dolbeault algebras on $Y$. There is a quasi-isomorphism $\iota_* \mathcal O_Y \to \iota_* \underline{\mathcal A_Y}$ (of $\mathcal O_X$-modules). By the equivalence $D^B_{perf}(\mathcal O_X) \simeq \mathcal P_{\mathcal A(X)}$ there is a cohesive module $T_Y$ over $\mathcal P_{A(Y)}$ and a quasi-isomorphism of $\mathcal O_X$-modules 
			\[ \iota_* \mathcal O_Y \simeq \underline{\mathcal A_X} \otimes_{\mathcal A(X)} T_Y .\]
			Now by the adjunction between $D(\mathcal O_X)$ and $D(\underline{\mathcal A_X})$ we obtain a morphsim \[t \in \text{Hom}_{D(\underline{\mathcal A_X})}( \underline{\mathcal A_X} \otimes_{\mathcal A(X)} T_Y, \iota_* \underline{\mathcal A_Y}) = \text{Hom}_{K(\underline{\mathcal A_X})}(\underline{\mathcal A_X} \otimes_{\mathcal A(X)} T_Y  , \iota_* \underline{\mathcal A_Y}).\]
			The $\mathcal A_X$-module $\underline{\mathcal A_X} \otimes_{\mathcal A(X)} T_Y$ locally lies in $Perf(\mathcal A_X)$ and by the softness of $\mathcal A^0(X)$ it follows that $\underline{\mathcal A_X} \otimes_{\mathcal A(X)} T_Y$ is $K$-projective. By the graded-projectivity of $T_Y$ we have $\Gamma(X,\underline{\mathcal A_X} \otimes_{\mathcal A(X)} T_Y) = T_Y$ and so $t$ yields a map $t(X) : T_Y \to \Gamma(X, \iota_* \underline{\mathcal A_Y}) = \mathcal A(Y)$ of $\mathcal A(X)$-modules. 
			
			We claim that $t(X)$ is a Fréchet twist-equivalence. By Lemma \ref{LemmaLocalReductionFre} it suffices to show the claim locally. Now note that a Fréchet twist-equivalence between cohesive modules is always an $A$-linear homotopy equivalence and so the choice of $T_Y$ is not important. Moreover, the property of being Fréchet acyclic is stable under taking tensor-products with Fréchet modules (by definition). Hence, the claim can be reduced to the simplest case of a skyscraper $0 \in U \subset \mathbb C$ which is the precisely the content of Corollary \ref{CorSkyscraperIsFreAcyclic}.
		\end{proof}
	\end{proposition}
	\begin{corollary}[Classical higher direct image Grauert theorem for compact manifolds]
		Let $f : X \to Y$ be a holomorphic map of compact complex manifolds. Then $\mathcal A(f) : \mathcal A(Y) \to \mathcal A(X)$ induces an $A_\infty$-functor
		\[ Rf_* : \mathcal P_{A(X)} \to \mathcal P_{A(Y)}\]
		and thus a functor on $H^0(Rf_*) : D^b_{coh}(X) \to  D^b_{coh}(Y)$ on homotopy categories.
		\begin{proof}
			By Proposition \ref{PropositionFréchetTwist}, $\mathcal A(Y)$ is Fréchet-twist smooth. The statement thus follows from Proposition \ref{PropositionGrauertMaps}.
		\end{proof}
	\end{corollary}
	
	\section{Derived categories of generalized complex tori}
     \subsection{Generalized Dolbeault algebras}
Consider a vector space $V$ together with a lattice $\Gamma \subset V$, complex structure $I : V \to V$ and a (real translation invariant) bivector $\beta \in \Wedge^2 V$ and a (real translation invariant) 2-form $B \in \Wedge^2 V^\vee$. To this data we can assign \emph{the generalized Dolbeault algebra}
\[ \mathcal A(I,\beta,B) := (C^\infty(V/\Gamma) \otimes \Wedge^{(0,*)} V_{\mathbb C}^\vee, \star^{\beta}, \bar \partial, 2 \pi i B^{0,2})\]
where the multiplication $\star^{\beta}$ is defined in terms of Fourier-modes 
\begin{align}\label{starProduct} (u_{\gamma_1} \otimes \omega_1) \star (u_{\gamma_2} \otimes \omega_2) := e^{\pi i \beta(\gamma_1,\gamma_2)} u_{\gamma_1 + \gamma_2} \otimes \omega_1 \wedge \omega_2\end{align}
for $u_{\gamma}(x) := e^{2 \pi i \gamma(x)}$, $\gamma \in \Gamma^\vee$.
\begin{lemma}
    \label{NCToriAreFrFr}
	The cdg-algebra $\mathcal A(I,\beta,B)$ is Fréchet-Fredholm. In particular,
	the morphism complexes in $\mathcal P_{\mathcal A(I,\beta,B)}$ have finite-dimensional cohomology.
	\begin{proof}
		Since we did not change the differential, $\bar \partial$ still makes $C^\infty(V/\Gamma) \otimes \Wedge^{(0,*)} V_{\mathbb C}^\vee$ into a Fréchet-Fredholm (pre-)complex. Moreover, since $|e^{\pi i \beta(\gamma_1,\gamma_2)}| = 1$, a direct computation shows that the multiplication also remains of tame degree $0$.
	\end{proof}
\end{lemma}
It will be argued in \cite{DCatsTori}, on the basis of T-duality and Mirror Symmetry, that the dg-category of cohesive modules $\mathcal P_{\mathcal A(I,\beta,B)} $ is the derived category of the \emph{generalized complex structure}
\begin{align}
	\mathcal J_{I,B,\sigma} := e^{\beta} e^{B} \mathcal J_{I} e^{-B} e^{-\beta}.
\end{align}
where $\mathcal J_{I} = \begin{pmatrix}
	I &0 \\ 0 &-I^*
\end{pmatrix}$ denotes the generalized complex structure associated to $I$. In fact, we will prove in that loc. cit. that Theorem \ref{ThmNCFM} implies that $\mathcal P_{\mathcal A(I,\beta,B)}$ depends only on the generalized complex structure $\mathcal J_{I,B,\sigma}$, rather than on the choice of $I$, $B$ and $\beta$.
\begin{proposition}
	The cdg-algebra $\mathcal A(I,\beta,B)$ is Fréchet twist-smooth.
	\begin{proof}
		We consider the following cdg-subalgebra of $\mathcal A(I_1,\beta_1,B_1) \hat \otimes \mathcal A(I_1,\beta_1,B_1)^{op}$ given by
		\[ \mathcal A_- := (C^\infty(V / \Gamma) \otimes  \Wedge^* [(V_{\mathbb C}^{\vee})^{(0,1)_{I_1}} \oplus (V_{\mathbb C}^{\vee})^{(0,1)_{I_1}} ]), \wedge, \bar \partial, 2\pi i (B_1^{0,2} \otimes 1 - 1 \otimes B_1^{0,2})) \]
		where $C^\infty(V / \Gamma)$ is embedded into $C^\infty(V / \Gamma \times V/\Gamma, (\beta_1,-\beta_1))$ via $u_{\gamma^*} \mapsto u_{(\gamma^*,-\gamma^*)}$. We consider the cdg-module $\mathcal S_{x=0}$ over $\mathcal A_-$, introduced in section \ref{SkyscraperConstr}. By Lemma \ref{RefLemmaT}
		we have an isomorphism
		\[ [\mathcal A(I_1,\beta_1,B_1) \hat \otimes \mathcal A(I_1,\beta_1,B_1)^{op}] \hat \otimes_{\mathcal A_-} \mathcal S_{x = 0} \cong \mathcal A(I_1,\beta_1,B_1).\]
		of $\mathcal A(I_1,\beta_1,B_1) \hat \otimes \mathcal A(I_1,\beta_1,B_1)^{op}$-modules.
		Hence it suffices to show that $S_{x = 0}$ is Fréchet twist-perfect. The cdg-algebra $\mathcal A_-$ can naturally by made into a (soft) sheaf of cdg-algebras on $V/\Gamma$ whose curvature is locally trivialisable. Thus, $\mathcal A_-$ is locally a product of the Dolbeault algebra on $V/\Gamma$ and an exterior algebra. The result then follows from Proposition \ref{PropositionFréchetTwist} and a version of Lemma \ref{LemmaLocalReductionFre}.
	\end{proof}
\end{proposition}

\subsection{Construction of Heisenberg modules}
We will now construct Heisenberg modules in the following way. We let $V$ be a real $2n$-dimensional vector space, together with a real non-degenerate bivector $\sigma \in \Wedge^2 V$ and let $\omega^b = (\sigma^\sharp)^{-1}$ be its inverse symplectic form. We consider smooth non-commutative Torus $C^\infty(V/\Gamma, \sigma)$ on $\sigma$ where $\Gamma$ is a lattice in $V$. We have a Schauder basis consisting of $u_{\gamma}(v) = e^{2\pi i \gamma(v)}$ where $\gamma \in \Gamma^\vee$ and the non-commutative product is defined as
\[ u_{\gamma_1} \star u_{\gamma_2} := e^{\pi i \sigma(\gamma_1,\gamma_2)} u_{\gamma_1 + \gamma_2} \]
which yield the commutation relations
\[ u_{\gamma_1} \star u_{\gamma_2} = e^{2 \pi i \sigma(\gamma_1, \gamma_2)} u_{\gamma_2} \star u_{\gamma_1}.\]
We consider $(V^\vee, \sigma)$ as a symplectic vector space. To  construct the Heisenberg module for this Torus, we choose complementary Lagrangian subspaces $L,L' \subset (V^\vee, \sigma)$. Then we have the isomorphism $\rho : L' \to L^*, \rho(l) = \sigma^\flat(l)|_{L}$ and we get an isomorphism of symplectic vector spaces
\begin{align} \label{Defphi} \phi : (L \oplus L^*, \omega_{\langle \cdot, \cdot \rangle})  \cong (V^\vee, \sigma), \phi(l,\alpha) := l + \rho^{-1}(\alpha) \end{align}
where $\omega_{\langle \cdot, \cdot \rangle}( (l,\alpha), (l',\alpha') )= \alpha(l') - \alpha'(l)$.
We consider the Heisenberg Lie-algebra $\mathfrak h$ as a central extension
\[ 0 \to \mathbb C \to \mathfrak h \xrightarrow{p} L \oplus L^* \to 0 \]
such that for $v,w \in \mathfrak h$:
\[ [v,w] = 2\pi i \omega_{\langle \cdot, \cdot \rangle}(p(v),p(w)).\]
The Heisenberg Lie-algebra $\mathfrak h$ acts on $\mathcal S(L)$. For a given splitting $r$ of $p$ the action is defined as
\begin{align} \label{firstaction} (r((l,\alpha)).g)(v) =  2 \pi i \alpha(v) g(v) - \partial_{l} g (v).\end{align}
We will also use the short notation $l.g + \alpha.g = (l,\alpha).g := r((l,\alpha)).g$. This extends to an action of $C^\infty(V/\Gamma, \sigma)$ by the formula
\begin{equation} \label{ExponentiatedAction} u_{\gamma}.g := e^{\phi^{-1}(\gamma)} g \end{equation}
for $\gamma \in \Gamma^\vee \subset V^\vee$. Denote the resulting Heisenberg module by $H(\sigma)$. We also define a connection on $H(\sigma)$ by
\begin{align} \label{defNabla} \nabla : V \times \mathcal S(L) \to \mathcal S(L), v.g := (\phi^{-1} \omega^\flat)(v) .g, \end{align}
indeed
\begin{align*} \nabla_v (u_\gamma g) &= (\phi^{-1} \omega^\flat)(v).(e^{\phi^{-1}(\gamma)} .g) \\&= e^{\phi^{-1}(\gamma)}.((\phi^{-1} \omega^\flat)(v).g) + 2\pi i \omega_{\langle \cdot, \cdot \rangle}((\phi^{-1} \omega^\flat)(v), \phi^{-1}(\gamma)) u_\gamma.g\\
	&= u_\gamma. \nabla_v(g) + 2\pi i \gamma(v) u_\gamma.g = u_\gamma. \nabla_v(g) + d_v(u_\gamma). g
\end{align*}
And similarly, one computes the curvature of this connection to be
\begin{align} \label{curvatureEq} [\nabla_v, \nabla_w].g = -2\pi i \omega(v,w) g.\end{align}

Equivalently, one can view $\nabla$ as a derivation of degree $1$ on $\mathcal S(L) \otimes \Wedge^* V^\vee$, making it into a cohesive cdg-module over the non-commutative curved de-Rham algebra \[\Omega(V/\Gamma, \sigma,-\omega) = (C^\infty(V/\Gamma) \otimes \Wedge^* V^\vee, \star^{\sigma},d_{DR},-2\pi i \omega) .\]

\begin{remark}
    One has to be cautious when interpreting formula \ref{ExponentiatedAction}. For $\alpha \in L^*$, we simply let
    \[ (e^\alpha.g)(x) := e^{2\pi i \alpha(x)} g(x).\]
    On the other hand, for $l \in L$ we have
    \[ (e^l.g)(x) = (e^{-\partial_l} g)(x) = \sum_{k = 0}^\infty \frac{(-\partial_l)^k}{k!} g(x)\]
    which is not convergent for the standard set of seminorms on $\mathcal S(L)$. However, the above formula can be viewed as the asymptotic expansion of $\hbar \mapsto g(x-\hbar l)$ at $\hbar = 1$. Thus we \emph{define} $(e^l.g)(x) := g(x-l)$. For general $\alpha + v$, the formula for $e^{\alpha + v}.g$ is deduced from the Baker-Campbell-Hausdorff formula (using that $[\alpha,v]$ is central):
    \[ (e^{\alpha + v}.g)(x) = e^{[\alpha,v]/2} (e^{v}.(e^\alpha).g)(x) = e^{\pi i \alpha(2x-v)}g(x-v)\]
\end{remark}

\subsubsection{Diagonal case}
\label{diagonalcase}
It will be important to us to consider the special case where $V$ has the form \[(V,\sigma) = (W, \sigma_W) \oplus (W, -\sigma_W).\]
Then we make the following choice of Lagrangian subspaces \[L = \{ (w,-w) | w \in W^*\} \subset W^* \times W^* , L' = \{ (w,w) | w \in W^*\} \subset W^* \times W^*.\] The isomorphism \ref{Defphi} becomes
\[ \phi : L \oplus L^* \to W^* \oplus W^*, \phi( (w,-w) \oplus (\alpha, -\alpha) ) = (w,-w) + (\omega(\alpha),\omega(\alpha)).\]
It is also useful to introduce the notation $x^+ := (x,x)$ and $x^- = (x,-x)$ so that $\phi ( w^-, \alpha^- ) = w^- + \omega^\sharp(\alpha)^+$. 
Note that in this notation, $\omega_{\langle \cdot,  \cdot \rangle}( w^-, \alpha^-  ) = - \alpha^-(w^-) = -2 \alpha(w).$

\subsection{Proof of non-commutative Fourier-Mukai duality}
As before, we fix a real vector space $V$ of dimension $2n$ together with a lattice $\Gamma \subset V$. The dual lattice is defined by
\[ \Gamma^\vee := \{ \gamma \in V^\vee  |  \gamma(\Gamma) \subseteq \mathbb Z \} \subseteq V^\vee \]
We consider the following data:
\begin{itemize}
	\item Real bivectors $\beta_1 \in \Wedge^2 V$ and $\beta_2 \in \Wedge^2 V^\vee$ (viewed as a bivector on $V^\vee$).
	\item Complex structures $I_1 : V \to V$, $I_2 : V^\vee \to V^\vee$.
	\item $2$-forms: $B_1 \in \Wedge^2 V^\vee$ and $B_2 \in \Wedge V^2$ (where we identified $(V^\vee)^* = V$).
\end{itemize}
fulfilling
\begin{itemize}
	\label{fulfilling}
	\item The map $\xi := 1-\beta_2^\sharp \beta_1^\sharp : V^\vee \to V^\vee$ is invertible.
	\item We have \begin{nalign}\label{AssEqs}
		&I_2 \xi= - \xi I_1^* \\
		&(B_1^\flat)^{(0,2)_{I_1}} = (\beta_2^\sharp (\xi^*)^{-1})^{(0,2)_{I_1}} \\
		&(B_2^\flat)^{(0,2)_{I_2}} = (\beta_1^\sharp \xi^{-1})^{(0,2)_{I_2}}.
	\end{nalign}
\end{itemize}
From these data we obtain non-commutative complex Tori
\[ \mathcal A(I_1,\beta_1,B_1) = (C^\infty(V/\Gamma) \otimes \Wedge^{0,*} V_{\mathbb C}^\vee, \star^{\beta_1}, \bar \partial_{I_1}, 2\pi i B_1^{0,2})\]
and 
\[ \mathcal A(I_2,\beta_2,B_2) = (C^\infty(V^\vee/\Gamma^\vee) \otimes \Wedge^{0,*} V_{\mathbb C}, \star^{\beta_2}, \bar \partial_{I_2}, 2\pi i B_2^{0,2}).\]

On $V \times V^\vee$ we have the bivector \[\sigma^\sharp = \begin{pmatrix}
	\beta_1^\sharp & -\text{id} \\ \text{id} & -\beta_2^\sharp 
\end{pmatrix} : V^\vee \oplus V = (V \oplus V^\vee)^* \to V \oplus V^\vee \]
whereas on $V^\vee \times V$ we define
\[\rho^\sharp = \begin{pmatrix}
	\beta_2^\sharp & -\text{id} \\ \text{id} & -\beta_1^\sharp 
\end{pmatrix} : (V^\vee \oplus V)^\vee = V \oplus V^\vee \to V^\vee \oplus V.\]
Note that $\rho$ is $-\sigma$ up to swapping $V$ and $V^\vee$.
These bivectors are non-degenerate because $(1-\beta_2^\sharp \beta_1^\sharp )$ is invertible. In fact, the inverse of $\sigma$ is
\begin{align}
	\label{equation:sigma} -\omega^\flat := -(\sigma^\sharp)^{-1} = \begin{pmatrix}
		\beta_2^\sharp (\xi^*)^{-1}  & \xi^{-1}
		\\ (\xi^*)^{-1} & -\beta_1^\sharp \xi^{-1}.
	\end{pmatrix}  : V \oplus V^\vee \to V^\vee \oplus V
\end{align}

Since the off-diagonal terms of $\sigma$ and $\rho$ are integral, one has isomorphisms of cdg-algebras \begin{align*} &\mathcal A(I_1 \times I_2, \sigma, B_1 -  B_2) = \mathcal A(I_1,\beta_1,B_1) \hat \otimes \mathcal A(I_2,\beta_2,B_2)^{op}  \\
	&  \mathcal A(I_2 \times I_1, \rho, B_2 -  B_1) = \mathcal A(I_2,\beta_2,B_2) \hat \otimes \mathcal A(I_1,\beta_1,B_1)^{op}\end{align*}

\begin{proposition}
	\label{PropNonCM}
	The Heisenberg module $H(\sigma)$ together with its connection yields a cohesive module $P_1$ over $\mathcal A(I_1 \times I_2,\sigma,B_1-B_2)$ whereas the Heisenberg module $H(\rho)$ defines a cohesive module $P_2$ for $\mathcal A(I_2 \times I_1, \rho, B_2 -  B_1)$. 
	\begin{proof}
		According to \ref{curvatureEq}, the curvature of the connection constructed on $H(\sigma)$ is given by formula \ref{equation:sigma}. To define a module over $\mathcal A(I_1 \times I_2,\sigma,B_1 \otimes 1- 1 \otimes B_2)$ we must have 
        \begin{align} \label{condOmega} \omega^{(0,2)} = (B_1 \otimes 1- 1 \otimes B_2)^{(0,2)} \end{align}
        where the $(0,2)$-part is computed with respect to $I_1 \times I_2$. Equation \ref{condOmega} is equivalent to the set of equations \ref{AssEqs}.
	\end{proof}
\end{proposition}
\begin{def1}
	The cohesive cdg-modules $P_1$ and $P_2$ appearing in Proposition \ref{PropNonCM} will be referred to as the \emph{non-commutative Poincaré line bundles}.
\end{def1}

We will show that the non-commutative Poincaré line bundles yield Fréchet-twist equivalences
\begin{align*}
		P_1 \hat \otimes_{\mathcal A(I_2,\beta_2,B_2)} P_2 \simeq \mathcal A(I_1,\beta_1,B_1)[n] \text{ as } \mathcal A(I_1,\beta_1,B_1) \hat \otimes \mathcal A(I_1,\beta_1,B_1)^{op}\text{-modules}\\
		P_2 \hat \otimes_{\mathcal A(I_1,\beta_1,B_1)} P_1 \simeq \mathcal A(I_2,\beta_2,B_2)[n] \text{ as } \mathcal A(I_2,\beta_2,B_2) \hat \otimes \mathcal A(I_2,\beta_2,B_2)^{op}\text{-modules}
\end{align*}
amounting to our main result, Theorem \ref{ThmNCFM}. The proof is technical, due to non-linearities already apparent in the relations \ref{AssEqs}. In fact, many formulas are considerably simpler either when $B_1 = 0$ or $\beta_1 = 0$. Proving the theorem will occupy the rest of the paper. To state the essential line of reasoning, let us give a proof sketch first:
\begin{proof}[Structure of the proof of Theorem \ref{ThmNCFM}]
\begin{enumerate}
    \item $P_1 \hat \otimes P_2$ defines a cohesive module over $\mathcal A(I_1 \times I_2,\sigma,B_1-B_2) \hat \otimes \mathcal A(I_2 \times I_1, \rho, B_2 -  B_1)$ and in fact it is isomorphic to the Heisenberg-module associated to the bivector $\sigma \otimes 1 + 1 \otimes \rho$.
    \item After applying canonical isomorphisms and Fourier-duality, the structure constants of this Heisenberg module are computed in Lemma
    \ref{LemmaAllEqs}.
    \item To get a module $\tilde T$ over $\mathcal A(I_1 \times I_1,\beta_1 \otimes 1 - 1 \otimes \beta_1, B_1 \otimes 1 - 1 \otimes B_1)$ we need to divide out the diagonal action by $\mathcal A(I_2,\beta_2,B_2)$ which amounts to the cdg-module described in Lemma \ref{FormulasLemma}.
    \item There is \emph{commutative} cdg-subalgebra 
    \[ \mathcal A_- := (C^\infty(V / \Gamma) \otimes  \Wedge^* [(V_{\mathbb C}^{\vee})^{(0,1)_{I_1}} \oplus (V_{\mathbb C}^{\vee})^{(0,1)_{I_1}} ]), \wedge, \bar \partial, 2\pi i (B_1^{0,2} \otimes 1 - 1 \otimes B_1^{0,2})) \]
    of $\mathcal A(I_1 \times I_1,\beta_1 \otimes 1 - 1 \otimes \beta_1, B_1 \otimes 1 - 1 \otimes B_1)$. The key point is that the bimodule $\mathcal A(I_1,\beta_1,B_1)$ can be identified as the pushforward of a certain "skyscraper" sheaf $S_{x = 0}$. On the other hand, $\tilde T$ can also be described as the pushforward of a subspace of $\mathcal R \subset \tilde T$ that is invariant under the action of $\mathcal A_-$.
    \item To compare $\tilde T \cong P_1 \hat \otimes_{\mathcal A(I_2,\beta_2,B_2)} P_2 $ with $\mathcal A(I_1,\beta_1,B_1)[n]$ it thus suffices to show that $\mathcal R$ is Fréchet twist-equivalent to the skyscraper $S_{x=0}$ which, as argued in Proposition \ref{PropCruc}, can be reduced to the claim that the sequence
    \[ 0 \to C^\infty(\mathbb C) \xrightarrow{\cdot z}  C^\infty(\mathbb C) \to \mathbb C[[\bar z]] \to 0\]
    is Fréchet twist-acyclic.
    
\end{enumerate}
\end{proof}

In particular, we can also consider their tensor product $P_1 \hat \otimes_{\mathbb C} P_2$ as a module over $\mathcal A(I_1 \times I_2,\sigma,B_1-B_2) \hat \otimes \mathcal A(I_2 \times I_1, \rho, B_2 -  B_1)$. It is useful also to swap factors so that 
$\mathcal A(I_2 \times I_1, \rho, B_2 -  B_1) = \mathcal A(I_1 \times I_2, \sigma, B_1 - B_2)^{op}$. The Lagrangian $L$ constructed in \ref{diagonalcase} is then given by \begin{align*}&L = L_0 \oplus L_1 \subset (V \oplus V^\vee \oplus V^\vee \oplus V)^\vee = V^\vee \oplus V \oplus V \oplus V^\vee \\ &L_0 = \{ (\alpha,0,0,-\alpha) = \alpha^- \}, L_1 = \{ (0,v,-v,0) = v^- \}.\end{align*}The symplectic isomorphism $\phi$ takes the form
\begin{align*} &\phi : (L_0 \oplus L_1) \oplus (L_0^* \oplus L_1^*) \to (V^\vee \oplus V \oplus V \oplus V^\vee, \sigma + \rho) \\
	&\phi( (\alpha^-, v^-), (\tilde v^-, \tilde \alpha^-)) = \alpha^- + v^- + (\sigma^{-1}(\tilde v,\tilde \alpha))^+.
\end{align*}
We will furthermore use the isomorphisms $\psi_0 : V^\vee \to L_0, \psi(\alpha) = (\alpha,0,0,-\alpha)$ and $\psi_1 : V \to L_1, \psi_1(v) = (0,v,-v,0)$, so that $\mathcal S(L_0 \oplus L_1) \cong \mathcal S(V^\vee \oplus V)$.

\begin{lemma}
	\label{Action}
	Under $\phi, \psi_0$ and $\psi_1$, the symplectic vector space $(V^\vee \oplus V \oplus V \oplus V^\vee, \sigma + \rho)$ acts on $\mathcal S(V^\vee \oplus V)$ as follows:
	\begin{align*}
		(\alpha^- + v^-).g(x) &= (\alpha,v,-v,-\alpha).g(x) = -\partial_{\alpha + v}g(x) \\
		(\alpha^+ + v^+).g(x) &= (\alpha,v,v,\alpha).g(x) = 4\pi i \sigma((\alpha,v), (x_0,x_1))  g(x) \\
		&= 4\pi i [ x_0(-v) + \alpha(x_1) + \beta_1(\alpha, x_0) - \beta_2(v,x_1)] g(x) 
	\end{align*}
	where $x = (x_0, x_1) = V^\vee \oplus V$, $g \in \mathcal S(V^\vee \oplus V)$.
    \begin{proof}
        The action is defined in equation \ref{firstaction}. Then apply the isomorphisms $\phi, \psi_0$ and $\psi_1$.
    \end{proof}
\end{lemma}

It will be useful to consider the Fourier-Transformation
\begin{align*}
	&\mathcal{F} : \mathcal S(V^\vee) \xrightarrow{\cong} \Wedge^{top} V^\vee_{\mathbb C} \otimes \mathcal S(V) \\
	&\mathcal F(g)(v) := ( \Wedge^{top} V_{\mathbb C} \ni P \mapsto  \int_{V^\vee} g(\alpha) e^{2 \pi i \alpha(v)} \Wedge P)
\end{align*}
Recall the identities for fixed $\alpha_0 \in V^\vee$ and $v_0 \in V$:
\begin{nalign}
\label{FTproperties}
	&\mathcal F( -\partial_{\alpha_0} g ) (v) = 2\pi i \alpha_0(v) g(v) \\
	&\mathcal F( 2\pi i v_0 \cdot g ) (v) = \partial_{v_0} g(v).
\end{nalign}
We thus get an isomorphism $\mathcal F \otimes \text{id} : \mathcal S(V^\vee \oplus V) \cong  \Wedge^{top} V^\vee_{\mathbb C} \otimes \mathcal S(V \oplus V)$.

\begin{lemma}
	\label{LemmaAllEqs}
	Under the isomorphism $\mathcal F \otimes \text{id} : \mathcal S(V^\vee \oplus V) \cong  \Wedge^{top} V^\vee_{\mathbb C} \otimes \mathcal S(V \oplus V)$, the action in Lemma \ref{Action} of $(V^\vee \oplus V \oplus V \oplus V^\vee, \sigma + \rho)$ on $\Wedge^{top} V^\vee_{\mathbb C} \otimes \mathcal S(V \oplus V)$ becomes
	\begin{align}
		(\alpha^- + v^-).g(x) &= -\partial_{(0,v)}g(x) + 2\pi i \alpha(x_0) g(x) \\
		(\alpha^+ + v^+).g(x) &= 4\pi i [ \alpha(x_1) - \beta_2(v,x_1)] g(x) \\
		&\quad + 2\partial_{(-v + \beta_1^\sharp(\alpha),0)} g(x)
	\end{align}
	where $x = (x_0,x_1) \in V \oplus V$.
	The exponentiated action is
	\begin{nalign}
		\label{ActionformulaFull}
		&e^{\alpha^- + v^- + \tilde \alpha^+ + \tilde  v^+}.g(x^-) = \exp(2\pi i (\alpha,2\tilde \alpha - 2\beta_2^\sharp(\tilde v)) + \partial_{(-2\tilde v + 2\beta_1^\sharp(\tilde \alpha),-v)} ).g(x) \\
		&= e^{2\pi i [\alpha(x_0 + \tilde v) + \tilde \alpha(2x_1 + v) - \beta_2(\tilde v, 2x_1 + v) - \beta_1(\tilde \alpha, \alpha)] } g (x_0 - 2\tilde v + 2\beta_1^\sharp(\tilde \alpha), x_1 - v).
	\end{nalign}
	In particular,
	\begin{align*}
		e^{\alpha^- + v^-}.g(x) = e^{2\pi i \alpha(x_0)} g(x_0, x_1 - v).
	\end{align*}
	The connection is given by 
	\begin{align*} 
		&\nabla : V \times V^\vee \times V^\vee \times V \otimes M \to M\\
		&\nabla_{\alpha^+ + v^+}g(x) = \nabla_{(v,\alpha,\alpha,v)} g(x) = \omega^\flat(v,\alpha)^-.g(x) \\
		&= \partial_{(0, (\xi^*)^{-1} v - \beta_1^\sharp \xi^{-1} \alpha)} g + 2 \pi i [	-\beta_2^\sharp (\xi^*)^{-1}v + \xi^{-1} \alpha](x_0) g(x). \\
		&\nabla_{\alpha^- + v^- = (v,\alpha,-\alpha,-v)} g(x) = \omega^\flat(v,\alpha)^+.g(x) = 2\partial_{(v,0)} g(x) + 4 \pi i \alpha(x_1) g(x)
	\end{align*}
	\begin{proof}
	    This follows from Lemma \ref{Action} and the properties of Fourier-Transform as in \ref{FTproperties}. The relation between the formula for the connection and the action is given in equation \ref{defNabla} whereas the formula for $\omega$ is given in \ref{equation:sigma}.
	\end{proof}
\end{lemma}
Note that the exponentiated action determines the action of $C^\infty(V/\Gamma_{\beta_1}  \times V^\vee/\Gamma^\vee_{-\beta_2}  \times V^\vee/\Gamma^\vee_{\beta_2}  \times V/\Gamma_{-\beta_1})$ on  $\Wedge^{top} V^\vee_{\mathbb C} \otimes \mathcal S(V \oplus V)$. Moreover, the commutative subalgebra spanned by $u_{\gamma^- + (\gamma^*)^-} = u_{(\gamma^*, \gamma, -\gamma, -\gamma^*)}$ acts in the following simple way:
\begin{nalign}
	\label{eqDiagonalAction}
	(u_{(\gamma^*, \gamma, -\gamma, -\gamma^*)}).g(x) &= e^{2\pi i \gamma^*(x_0)} g( x_0, x_1 - \gamma )
	\\ &= u_{\gamma^*}(x_0)  g( x_0, x_1 - \gamma)
\end{nalign}

The smooth Heisenberg module $H(\sigma,\rho) \cong \Wedge^{top} V^\vee_{\mathbb C} \otimes \mathcal S(V \oplus V)$ endowed with the $(0,1)$-part of $\nabla$ becomes a cohesive cdg-module over $\mathcal A(I_1 \times I_2,\sigma,B_1-B_2) \hat \otimes \mathcal A(I_2 \times I_1, \rho, B_2 -  B_1)$ of the form \[ \mathcal T := (\Wedge^{top} V^\vee_{\mathbb C} \otimes \mathcal S(V \oplus V) \otimes \Wedge [(V_{\mathbb C}^{\vee})^{(0,*)_{I_1}} \oplus V_{\mathbb C}^{(0,1)_{I_2}} \oplus V_{\mathbb C}^{(0,1)_{I_2}} \oplus (V_{\mathbb C}^{\vee})^{(0,1)_{I_1}} ], \nabla^{0,1} ) \]

\begin{lemma}
	\label{FormulasLemma}
	The quotient of the cohesive  cdg-module $\mathcal T$ over $\mathcal A(I_1 \times I_2,\sigma,B_1-B_2) \hat \otimes \mathcal A(I_2 \times I_1, \rho, B_2 -  B_1)$ by the diagonal action of $\mathcal A(I_2,\beta_2,B_2)$ is equal to the \\$\mathcal A(I_1,\beta_1,B_1) \hat \otimes \mathcal A(I_1,\beta_1,B_1)^{op}$-module
	\[ \tilde{\mathcal T} := \mathcal T / \mathcal A(I_2,\beta_2,B_2) = (\mathcal T_0 \otimes \Wedge^{(0,*)_{I_1}} (V_\mathbb C^* \oplus V_\mathbb C^*), \mathbb E_0 + \mathbb E_1) \]
	where \[ \mathcal T_0 = \Wedge^{top} V^\vee_{\mathbb C} \otimes \mathcal S(V) \hat \otimes C^\infty(V/\Gamma) \otimes \Wedge^{(0,*)_{I_2}} V_{\mathbb C} \]
	is endowed with the action
	\begin{nalign} \label{EqActionQ} (u_{(\gamma_1^*,\gamma_2^*)}.g)(x_0, x_1) = &e^{2\pi i [ \frac{1}{2}(\gamma_1^* - \gamma_2^*)(x_0) + (\gamma_1^* + \gamma_2^*)(x_1) + \frac{1}{2} \beta_1(\gamma_1^*, \gamma_2^*)]} \cdot \\&\quad \cdot g( x_0 + \beta_1^\sharp(\gamma_1^* + \gamma_2^*), x_1)\end{nalign}
	and differentials
	\begin{align}\label{E0Def} \mathbb E_0 (g \otimes \omega)(x) := 2\pi i g(x) \wedge ((\xi^*)^{-1} x_0)^{(0,1)_{I_2}} \wedge \omega \end{align}
	and
	\begin{nalign} \label{EqE1} &\mathbb E_1 :  \mathcal T_0 \to \mathcal T_0  \otimes \Wedge^{(0,1)_{I_1}} (V_\mathbb C^* \oplus V_\mathbb C^*) \\
		&(-1)^{|\omega|} \mathbb E_1( g \otimes \omega) (v_-^{0,1} + \tilde v_+^{0,1}) = \iota_{v_-^{0,1} + \tilde v_+^{0,1}} (\mathbb E_1( g \otimes \omega)) = \nabla_{v_-^{0,1} + \tilde v_+^{0,1}}(g) \otimes \omega
	\end{nalign}
	for $g \in \Wedge^{top} V^\vee_{\mathbb C} \otimes \mathcal S(V) \hat \otimes C^\infty(V/\Gamma), \omega \in \Wedge^{(0,*)_{I_2}} V_{\mathbb C}$.
	\begin{proof}
		We need to identify the cokernel of 
		\begin{align*}&Q : \mathcal A(I_2,\beta_2,B_2)\hat \otimes \mathcal T \to \mathcal T  \\
			&Q(\omega,g) := (1 \otimes \omega \otimes 1 \otimes 1).g - (1 \otimes 1 \otimes \omega \otimes 1).g.
		\end{align*}
		It has the form
		\begin{align*}
			\Wedge^{top} V^\vee_{\mathbb C} \otimes \mathcal S(V) \hat \otimes \mathcal S(V) / \mathcal S(\Gamma) \otimes \Wedge [(V_{\mathbb C}^{\vee})^{(0,*)_{I_1}} \oplus V_{\mathbb C}^{(0,1)_{I_2}}  \oplus (V_{\mathbb C}^{\vee})^{(0,1)_{I_1}} ]
		\end{align*}
		where $\gamma \in \Gamma$ acts on $g \in \mathcal S(V)$ by $\gamma.g(x) = g(x - \gamma)$ in view of equation \ref{eqDiagonalAction}. Now we use the isomorphism
		\begin{align*}
			\mathcal S(V) / \mathcal S(\Gamma) \to C^\infty(V/ \Gamma), [g] \mapsto (x \mapsto \sum_{\gamma \in \Gamma} g(x + \gamma)).
		\end{align*}
		The inverse can be constructed by choosing a Schwartz function $h \in \mathcal S(V)$ satisfying $1 = \sum_{\gamma \in \Gamma} h(x + \gamma)$ for all $x$. In fact, such a choice determines a splitting \[C^\infty(V / \Gamma) \to \mathcal S(V), g \mapsto g \cdot h.\]
		The formulas for the action and connection follow from the corresponding formulas in \ref{LemmaAllEqs}.
	\end{proof}
\end{lemma}

We introduce a \emph{commutative} cdg-subalgebra of $\mathcal A(I_1,\beta_1,B_1) \hat \otimes \mathcal A(I_1,\beta_1,B_1)^{op}$ given by
\[ \mathcal A_- := (C^\infty(V / \Gamma) \otimes  \Wedge^* [(V_{\mathbb C}^{\vee})^{(0,1)_{I_1}} \oplus (V_{\mathbb C}^{\vee})^{(0,1)_{I_1}} ]), \wedge, \bar \partial, 2\pi i (B_1^{0,2} \otimes 1 - 1 \otimes B_1^{0,2})) \]
where $C^\infty(V / \Gamma)$ is embedded into $C^\infty(V / \Gamma \times V/\Gamma, (\beta_1,-\beta_1))$ via $u_{\gamma^*} \mapsto u_{(\gamma^*,-\gamma^*)}$. Consider the following subspace of $\tilde {\mathcal T}$
\begin{align}  \mathcal R &:= \Wedge^{top} V^\vee_{\mathbb C} \otimes \mathcal S(V) \hat \otimes \mathbb C \otimes \Wedge [(V_{\mathbb C}^{\vee})^{(0,1)_{I_1}} \oplus V_{\mathbb C}^{(0,1)_{I_2}}  \oplus (V_{\mathbb C}^{\vee})^{(0,1)_{I_1}}  ] \\
	& \subseteq \Wedge^{top} V^\vee_{\mathbb C} \otimes \mathcal S(V) \hat \otimes C^\infty(V / \Gamma) \otimes \Wedge [(V_{\mathbb C}^{\vee})^{(0,1)_{I_1}} \oplus V_{\mathbb C}^{(0,1)_{I_2}}  \oplus (V_{\mathbb C}^{\vee})^{(0,1)_{I_1}}]
\end{align}
where $\mathbb C$ is included into $C^\infty(V / \Gamma)$ via the unit.
\begin{lemma}
	The subspace $\mathcal R$ is invariant under the differential $\mathbb E_0 + \mathbb E_1$ and the action of $\mathcal A_-$. Hence, it defines a cohesive cdg-module over $\mathcal A_-$ and the inclusion $\mathcal R \to \tilde{\mathcal T}$ induces an isomorphism of cdg-modules over $\mathcal A(I_1,\beta_1,B_1) \hat \otimes \mathcal A(I_1,\beta_1,B_1)^{op}$:
	\[ [\mathcal A(I_1,\beta_1,B_1) \hat \otimes \mathcal A(I_1,\beta_1,B_1)^{op} ] \hat \otimes_{\mathcal A_-} \mathcal R \to \tilde {\mathcal T}.\]
	\begin{proof}
		The fact that $\mathcal R$ defines an $\mathcal A_-$-module follows from the formulas in Lemma \ref{FormulasLemma}. 
	\end{proof}
\end{lemma}

\subsubsection{Comparison with the diagonal}
\label{SkyscraperConstr}
We need to show that $\tilde T$ is Fréchet twist-equivalent to the diagonal bimodule $\mathcal A(I_1,\beta_1,B_1)$ over $\mathcal A(I_1,\beta_1,B_1) \hat \otimes \mathcal A(I_1,\beta_1,B_1)^{op}$. To do this, we first define the following "skyscraper" $\mathcal S_{x = 0}$ over the subalgeba $\mathcal A_-$. It has the form
\[ \mathcal S_{x = 0} = \Wedge ( V_{\mathbb C}^{+,*})^{(0,1)_{I_1}}  \]
where $V^+ = \{ (v,v) \in V \times V | v \in V\}$ and action
\[ \mu(\omega, \eta) = i^*(\omega(0)) \wedge \eta\]
for $\omega \in \mathcal A_-$, $i$ is the inclusion $i : V^+ \to V \oplus V$ and $\eta \in \mathcal S_{x = 0}$. It is endowed with the $0$-differential. Note that this is indeed a cdg-module as $i^*( B_1 \otimes 1 - 1 \otimes B_1 ) = 0$ so that the curvature acts trivially on $\mathcal S_{x = 0}$. 
The following Lemma is the reason for considering this module:
\begin{lemma}
	\label{RefLemmaT}
	There is an isomorphism of $\mathcal A(I_1,\beta_1,B_1) \hat \otimes \mathcal A(I_1,\beta_1,B_1)^{op}$-modules:
	\[ [\mathcal A(I_1,\beta_1,B_1) \hat \otimes \mathcal A(I_1,\beta_1,B_1)^{op}] \hat \otimes_{\mathcal A_-} \mathcal S_{x = 0} \cong \mathcal A(I_1,\beta_1,B_1).\]
    \begin{proof}
        We note that the tensor product need not be derived as $\mathcal A(I_1,\beta_1,B_1) \hat \otimes \mathcal A(I_1,\beta_1,B_1)^{op}$ is graded-free over $\mathcal A_-$ with fiber $C^\infty(V/\Gamma)$. The identity follows by computation of the left hand side. 
    \end{proof}
\end{lemma}

It will also be useful to introduce another cdg-module $\mathcal S'_{x = 0}$ of the form
\[ \mathcal S'_{x = 0} = \mathbb C[[ (V^\vee_{\mathbb C})^{(0,1)_{I_1}} ]]  \otimes \Wedge [(V_{\mathbb C}^{\vee})^{(0,1)_{I_1}} \oplus (V_{\mathbb C}^{\vee})^{(0,1)_{I_1}} ]\]
endowed with differential 
\begin{nalign}
	\label{DiffF1}
	&\mathbb F_1 : \mathbb C[[(V^\vee_{\mathbb C})^{(0,1)_{I_1}}]] \to  \mathbb C[[(V^\vee_{\mathbb C})^{(0,1)_{I_1}}]]  \otimes \Wedge^1 [(V_{\mathbb C}^{\vee})^{(0,1)_{I_1}} \oplus (V_{\mathbb C}^{\vee})^{(0,1)_{I_1}} ] \\
	&\mathbb F_1( f )(v_-^{0,1} + \tilde v_+^{0,1} ) = \nabla_{v_-^{0,1} + \tilde v_+^{0,1} } f = -2\pi i ( (B_1^\flat)^{0,2}\tilde v^{0,1}) \cdot f + 2\partial_{v^{0,1}} f
\end{nalign}
and action \begin{align} \label{Poweraction} f.g := T_{\bar \partial_{I_2},x = 0}(f) \wedge g. \end{align}
A direct computation shows that $(\mathbb F^1)^2= 2\pi i (B_1^{0,2} \otimes 1 - 1 \otimes B_1^{0,2})$.
\begin{lemma}
	\label{RefLemmaT2}
	The cdg-modules $\mathcal S'_{x = 0}$ and $\mathcal S_{x = 0}$ are isomorphic in $D^{ctr}(\mathcal A_-)$.
	\begin{proof}
		There is a natural map $\mathcal S'_{x = 0} \to \mathcal S_{x = 0}$ given by pulling back along $i : V^+ \to V \times V$ and evaluating the power series at $0$. On $\mathcal S'_{x = 0}$ one has the filtration
		\[ F^i \mathcal S'_{x = 0} =  \bigoplus_{a+b = i} [\mathbb C[[(V^\vee_{\mathbb C})^{(0,1)_{I_1}} ]]_{\geq a}  \otimes \Wedge (V_{\mathbb C}^{+,*})^{(0,1)_{I_1}} \otimes \Wedge^{\geq b} (V_{\mathbb C}^{-,*})^{(0,1)_{I_1}} ]\]
		whereas on $\mathcal S_{x = 0}$ we let $F^0 \mathcal S_{x = 0} = \mathcal S_{x = 0}, F^1 \mathcal S_{x = 0}  = 0$. All structure, including the map $\mathcal S'_{x = 0} \to \mathcal S_{x = 0}$,  is compatible with this filtration and the filtration is bounded below and complete. Furthermore it is easy to write down contracting homotopies for the quotients $F^i \mathcal S'_{x = 0} / F^{i+1} \mathcal S'_{x = 0}, i \geq 1$ and also one has $F^0 \mathcal S'_{x = 0} / F^{1} \mathcal S'_{x = 0} \cong \mathcal S_{x = 0}$. Thus, $\mathcal S'_{x = 0} \to \mathcal S_{x = 0}$ is an isomorphism in $D^{ctr}(\mathcal A_-)$.
	\end{proof}
\end{lemma}

\begin{proposition}
	\label{PropCruc}
	The cdg-modules $\Wedge^{n} (V_\mathbb C^{(1,0)_{I_2}})^*[n] \otimes \mathcal S'_{x = 0}$ and $\mathcal R$ are naturally Fréchet-twist equivalent. Consequently, $\mathcal S'_{x = 0}[n]$ is Fréchet twist-equivalent to $\mathcal R$, $2n = \dim_{\mathbb R}V$. 
	\begin{proof}
		Recall that $\mathcal R$ has the form
		\[ \mathcal R = \Wedge^{2n} V^\vee_{\mathbb C} \otimes \mathcal S(V) \hat \otimes \mathbb C \otimes \Wedge [(V_{\mathbb C}^{\vee})^{(0,1)_{I_1}} \oplus V_{\mathbb C}^{(0,1)_{I_2}}  \oplus (V_{\mathbb C}^{\vee})^{(0,1)_{I_1}} ] \] whereas 
		\[\mathcal S'_{x = 0} = \mathbb C[[(V^\vee_{\mathbb C})^{(0,1)_{I_1}} ]]  \otimes \Wedge [(V_{\mathbb C}^{\vee})^{(0,1)_{I_1}} \oplus (V_{\mathbb C}^{\vee})^{(0,1)_{I_1}} ].\]
		We use the contraction isomorphism \[\Wedge^{2n} V^\vee_{\mathbb C}  \otimes  \Wedge^* V_{\mathbb C}^{(0,1)_{I_2}} \cong \Wedge^{(n,n-*)_{I_2}} (V_\mathbb C)^*.\]
		so that $\mathcal R$ takes the form
		\[ \Wedge^{(n,n-*)_{I_2}} (V_\mathbb C)^*\otimes \mathcal S(V) \hat \otimes \mathbb C \otimes \Wedge [(V_{\mathbb C}^{\vee})^{(0,1)_{I_1}}  \oplus (V_{\mathbb C}^{\vee})^{(0,1)_{I_1}} ].\]
		We define the map $\tau : \mathcal R \to \Wedge^{n} (V_\mathbb C^{(1,0)_{I_2}})^*[n] \otimes \mathcal S'_{x = 0}$ to be
		\[ \tau(\omega \otimes g \otimes \eta) := \omega \otimes T_{\bar \partial_{I_2},0}(g) \otimes \eta\]
		where $\omega \in \Wedge^{n} (V_\mathbb C^{(1,0)_{I_2}})^*$, $g \in \mathcal S(V)$ and $\eta \in \Wedge [(V_{\mathbb C}^{\vee})^{(0,1)_{I_1}}  \oplus (V_{\mathbb C}^{\vee})^{(0,1)_{I_1}} ]$ and all other components are $0$. One has $\mathbb F_1 \circ \tau = \tau \circ \mathbb E_1$ comparing equation \ref{DiffF1} with \ref{EqE1} (using \ref{AssEqs}). Moreover, the dependence of \[((\xi^*)^{-1} x_0)^{(0,1)_{I_2}} = (\xi^*)^{-1} x_0^{(1,0)_{I_1}}\] on $x_0$ is holomorphic (with respect to $I_1$) by equation \ref{AssEqs}. By definition of $\mathbb E_0$ (\ref{E0Def}), this implies that $\tau \circ \mathbb E_0 = 0$.
		The map $\tau$ also commutes with actions by comparing \ref{EqActionQ} with the definition of the action on $\mathcal S'_{x = 0}$, equation \ref{Poweraction}. 
		
		To show that $\text{cone}(\tau)$ is Fréchet twist-acyclic, we consider the sequence
		\[ 0 \to \Wedge^{(0,n)_{I_2}} (V_\mathbb C)^* \otimes \mathcal S(V) \to \Wedge^{(0,n-1)_{I_2}} (V_\mathbb C)^* \otimes \mathcal S(V) \to ... \to \mathcal S(V) \to \mathbb C[[(V^\vee_{\mathbb C})^{0,1}]] \to 0 \]
		where the differential is induced by $\mathbb E_0$ (see Lemma \ref{FormulasLemma}) and $T_{\bar \partial_{I_2},0}$ on the right. This sequence is isomorphic to an $n$-fold projective tensor product of the sequence
		\[ 0 \to \mathcal S(\mathbb C) \xrightarrow{\cdot z} \mathcal S(\mathbb C) \to \mathbb C[[\bar z]] \to 0\]
		which is exact and stably so under tensoring with any Banach space, by Corollary \ref{CorBanachStable}. From Lemma \ref{LemmaFreAcyclic} we obtain that $\text{cone}(\tau)$ is Fréchet twist-acyclic.
	\end{proof}
\end{proposition}

\begin{proposition}
	\label{PropNCFM}
	Let $P_1$ and $P_2$ be the non-commutative Poincaré line bundles over $\mathcal A(I_1 \times I_2,\beta_1 - \beta_2,B_1-B_2)$ and $\mathcal A(I_1 \times I_2,\beta_1 - \beta_2,B_1-B_2)^{op}$, where $I_1,I_2,\beta_1,\beta_2,B_1,B_2$ are subject to the relations \ref{fulfilling}. Then there are Fréchet twist-equivalences
	\begin{align*}
		P_1 \hat \otimes_{\mathcal A(I_2,\beta_2,B_2)} P_2 \simeq \mathcal A(I_1,\beta_1,B_1)[n] \text{ as } \mathcal A(I_1,\beta_1,B_1) \hat \otimes \mathcal A(I_1,\beta_1,B_1)^{op}\text{-modules}\\
		P_2 \hat \otimes_{\mathcal A(I_1,\beta_1,B_1)} P_1 \simeq \mathcal A(I_2,\beta_2,B_2)[n] \text{ as } \mathcal A(I_2,\beta_2,B_2) \hat \otimes \mathcal A(I_2,\beta_2,B_2)^{op}\text{-modules}
	\end{align*}
	where $2n = \text{dim}_{\mathbb R}(V/\Gamma)$. 
	\begin{proof}
		We show that $P_1 \hat \otimes_{\mathcal A(I_2,\beta_2,B_2)} P_2 \simeq \mathcal A(I_1,\beta_1,B_1)[n]$. The other direction follows by symmetry. 
		By construction $P_1 \hat \otimes_{\mathcal A(I_2,\beta_2,B_2)} P_2$ is equal to $\tilde{\mathcal T}$, as given in Lemma \ref{FormulasLemma}.
		Moreover, Lemmas \ref{RefLemmaT}, \ref{RefLemmaT2} and Proposition \ref{PropCruc} yield a Fréchet twist-equivalence \[\tilde{\mathcal T} \simeq \mathcal A(I_1,\beta_1,B_1)[n], \]
		proving the claim.
	\end{proof}
\end{proposition}
From the proposition we obtain our main theorem as a corollary:
\begin{theorem}
    \label{ThmNCFM}
    The non-commutative Poincaré line bundles $P_1$ and $P_2$ over $\mathcal A(I_1 \times I_2,\beta_1 - \beta_2,B_1-B_2)$ and $\mathcal A(I_1 \times I_2,\beta_1 - \beta_2,B_1-B_2)^{op}$, where $I_1,I_2,\beta_1,\beta_2,B_1,B_2$ are subject to the relations \ref{fulfilling}, induce mutually inverse equivalences
    \begin{align*}
		&P_1 \hat \otimes (-) : D^{FrTw}_{Fré}({\mathcal A(I_1,\beta_1,B_1)}) \to  D^{FrTw}_{Fré}({\mathcal A(I_2,\beta_2,B_2)}) \\
		&P_2[-n] \hat \otimes (-) : D^{FrTw}_{Fré}({\mathcal A(I_2,\beta_2,B_2)}) \to D^{FrTw}_{Fré}({\mathcal A(I_1,\beta_1,B_1)})
	\end{align*} 
    which restrict to $A_\infty$-quasi-equivalences on the level of cohesive modules: 
    \begin{align*}
		&P_1 \hat \otimes (-) : \mathcal P_{\mathcal A(I_1,\beta_1,B_1)} \to \mathcal P_{\mathcal A(I_2,\beta_2,B_2)} \\
		&P_2[-n] \hat \otimes (-) : \mathcal P_{\mathcal A(I_2,\beta_2,B_2)} \to \mathcal P_{\mathcal A(I_1,\beta_1,B_1)}.
	\end{align*} 
    \begin{proof}
        By construction, the class of Fréchet twist-acyclic modules is stable under tensoring with other Fréchet modules. Hence, Proposition \ref{PropNCFM} implies that we obtain equivalences on the level of Fréchet modules localised along Fréchet twist-equivalences. 
        
        By Lemma \ref{LemmaCohesiveStable}, a cohesive module $M$ is Fréchet twist-acyclic if and only if it is $A$-homotopy equivalent to $0$ ( $\Leftrightarrow M \simeq 0$ in $\mathcal P_A$). Moreover, the functors induced by $P_1$ and $P_2$ restrict to cohesive modules by the Grauert-principle, Theorem \ref{GrauertPrinciple} and Lemma \ref{NCToriAreFrFr}.
    \end{proof}
\end{theorem}

\begin{corollary}
    Let $(X = V/\Gamma, I : V \to V)$ be a complex torus, $\beta \in \Wedge^2 V$ and $B \in \Wedge^2 V^\vee$ such that $\alpha := 1+B^\flat \beta^\sharp$ is invertible. Then for
    \begin{align*}
        &I_2 := -\alpha^{-1} I_1^* \alpha\\
        &B_2^\flat := \beta^\sharp \alpha \\
        &\beta^\sharp_2 := \alpha^{-1} B^\flat
    \end{align*}
    we have an $A_\infty$-quasi-equivalence
    \[\mathcal P_{\mathcal A(I,\beta,B)} \to \mathcal P_{\mathcal A(I_2,\beta_2,B_1)} \]
    where $\mathcal A(I_2,\beta_2,B_1)$ is taken with respect to $V^\vee / \Gamma^\vee$.
    In particular, if $B = 0$, we get an $A_\infty$-quasi-equivalence
     \[\mathcal P_{\mathcal A(I,\beta,0)} \to \mathcal P_{\mathcal A(-I^*, 0, B_{\beta})} \]
     where $B_\beta$ is the 2-form on $V^\vee$ defined by $B_\beta^\flat := \beta^\sharp$.
\end{corollary}

Consider an elliptic curve $X = (V/\Gamma, I : V\to V)$, $\text{dim}_{\mathbb R}(V) = 2$. From the point of view of Hochschild cohomology, $X$ has no non-commutative deformations, as \[0 = H^0(\Wedge^2 TX) \subseteq HH^2(X).\]
However, any \emph{real} bivector results in a deformed Dolbeault algebra $\mathcal A(I,\beta,0)$ with the star product defined as in \ref{starProduct}. It has been observed by Polishchuk and Schwarz \cite{polishchuk2003categories} that a certain subcategory of $D^b_{coh}$ can be matched between elliptic curves and their non-commutative analogues. We have the following generalizations:
\begin{corollary}
    Let $X = V/\Gamma$ be a complex torus and let $\beta,\beta' \in \Wedge^2 V$ be real bivectors. If $(\beta)^{(2,0)} = (\beta')^{(2,0)}$ then we have an $A_\infty$-quasi-equivalence \[ \mathcal P_{\mathcal A(I,\beta,0)} \to \mathcal P_{\mathcal A(I,\beta',0)}.\]
    \begin{proof}
        Define a 2-form $B$ on $V^\vee$ by the formula $B^\flat = \beta'^\sharp$. Hence \[(B^\flat)^{(0,2)_{-I^*}} = (\beta^\sharp)^{(0,2)_{-I^*}} = (\beta'^\sharp)^{(0,2)_{-I^*}}\] holds. The theorem tells us that
       \[ \mathcal P_{\mathcal A(I,\beta_1,0)} \simeq \mathcal P_{\mathcal A(-I^*,0,B)} \simeq \mathcal P_{\mathcal A(I,\beta_2,0)}.\]
    
    \end{proof}
\end{corollary}
\begin{corollary}[No non-commutative elliptic curves]
    \label{PolishStrength}
    Let $X = V/\Gamma$ be an elliptic curve and let $\beta \in \Wedge^2 V$ be a real bivector. Then there is an $A_\infty$-quasi-equivalence
    \[ \mathcal P_{\mathcal A(I,\beta,0)} \xrightarrow{\sim} \mathcal P_{\mathcal A(I,0,0)}.\]
\end{corollary}

\bibliographystyle{alphaurl}
\bibliography{NCFMGR}

\end{document}